\documentclass{amsart}
\usepackage{amsmath,amssymb,amsthm,mathrsfs}

\newtheorem{thm}{Theorem}[section]
\newtheorem{lem}[thm]{{Lemma}}
\newtheorem{cor}[thm]{{Corollary}}
\newtheorem{rem}[thm]{{Remark}}
\newtheorem{prop}[thm]{{Proposition}}

\theoremstyle{definition}
\newtheorem{defn}[thm]{{Definition}}

\numberwithin{equation}{section}

\newcommand{\sbm}[1]{\left[\begin{smallmatrix} #1
                \end{smallmatrix}\right]}

\newcommand{\C}{{\mathbb C}}
\newcommand{\D}{{\mathbb D}}
\newcommand{\T}{{\mathbb T}}
\newcommand{\Z}{{\mathbb Z}}
\newcommand{\AP}{{\bf AIP}_{\cH(K_S)}}
\newcommand{\APS}{{\bf AIP}_{\cS(\cU,\cY)}}
\newcommand{\OAP}{{\bf OAP}_{\cH(K_S)}}

\newcommand{\mat}[2]{\ensuremath{\left[\begin{array}{#1}#2\end{array} \right]}}

\newcommand{\ov}[1]{{\overline{#1}}}

\newcommand{\inn}[2]{\ensuremath{\langle #1,#2 \rangle}}
\newcommand{\tu}[1]{\textup{#1}}

\newcommand{\Ran}{\textup{Ran\,}}

\newcommand{\wtil}{\widetilde}
\newcommand{\half}{\frac{1}{2}}
\newcommand{\Red}{{\boldsymbol{\mathcal R}}}

\newcommand{\cD}{{\mathcal D}}
\newcommand{\cE}{{\mathcal E}}
\newcommand{\cH}{{\mathcal H}}

\newcommand{\cK}{{\mathcal K}}\newcommand{\cL}{{\mathcal L}}
\newcommand{\cM}{{\mathcal M}}\newcommand{\cN}{{\mathcal N}}
\newcommand{\cO}{{\mathcal O}}
\newcommand{\cR}{{\mathcal R}}
\newcommand{\cS}{{\mathcal S}}
\newcommand{\cU}{{\mathcal U}}
\newcommand{\cW}{{\mathcal W}}\newcommand{\cX}{{\mathcal X}}
\newcommand{\cY}{{\mathcal Y}}

\newcommand{\Ga}{\Gamma}
\newcommand{\de}{\delta}\newcommand{\De}{\Delta}

\newcommand{\Si}{\Sigma}

\newcommand{\N}{{\mathbb N}}

\newcommand{\bx}{{\mathbf x}}

\newcommand{\bstar}{{[*]}}

\newcommand{\FS}{{F^S}}
\newcommand{\FSstar}{{F^{S\bstar}}}

\newcommand{\Fs}{{F^s}}

\newcommand{\fD}{{\mathfrak D}}
\newcommand{\BC}{{\mathbb C}}
\newcommand{\bU}{{\mathbf U}}
\newcommand{\BD}{{\mathbb D}}

\numberwithin{equation}{section}

\begin{document}

\title[Interpolation in de Branges-Rovnyak spaces]
{Abstract interpolation in vector-valued \\ de Branges-Rovnyak spaces}

\author{Joseph A. Ball}
\address{Joseph A. Ball, Department of Mathematics,
Virginia Tech, Blacksburg, VA 24061-0123, USA}
\email{ball@math.vt.edu}

\author{Vladimir Bolotnikov}
\address{Vladimir Bolotnikov, Department of Mathematics,
The College of William and Mary,
Williamsburg VA 23187-8795, USA}
\email{vladi@math.wm.edu}

\author{S. ter Horst}
\address{S. ter Horst, Department of Mathematics, Unit for BMI,
North-West University, Potchefstroom 2531, South Africa}
\email{Sanne.TerHorst@nwu.ac.za}

\begin{abstract}
Following ideas from the Abstract Interpolation Problem of \cite{kky} for Schur class functions,
we study a general metric constrained interpolation problem for functions from a vector-valued de
Branges-Rovnyak space $\cH(K_S)$ associated with an operator-valued Schur class function $S$.
A description of all solutions is obtained in terms of functions from an associated de Branges-Rovnyak
space satisfying only a bound on the de Branges-Rovnyak-space norm.
Attention is also paid to the case that the map which provides this description is injective.
The interpolation problem studied here contains as particular cases (1) the vector-valued version of
the interpolation problem with operator argument considered recently in \cite{bbt2} (for the
nondegenerate and scalar-valued case) and (2) a boundary
interpolation problem in $\cH(K_S)$.
In addition, we discuss connections with results on kernels
of Toeplitz operators and nearly invariant subspaces of the backward shift operator.
\end{abstract}


\maketitle


\section{Introduction}

De Branges-Rovnyak spaces play a prominent role in Hilbert space approaches to
$H^\infty$-interpolation. However, very little work exists on interpolation for
functions in de Branges-Rovnyak spaces themselves. In this paper we pursue our
studies of interpolation problems for functions in de Branges-Rovnyak spaces,
which started in \cite{bbt2}. We consider a norm constrained interpolation problem
(denoted by $\AP$ in what follows), which is sufficiently fine so as to include
on the one hand interpolation problems with operator argument (considered
for the nondegenerate and scalar-valued case in \cite{bbt2}) and, on the other hand,
boundary interpolation problems; it is only recent work \cite{bk3, bk6, bk5, fm}
which has led to a systematic understanding of boundary-point evaluation on de
Branges-Rovnyak spaces from an operator-theoretic point of view.

In order to state the interpolation problem we first introduce some definitions
and notations. As usual, for Hilbert spaces $\cU$ and $\cY$ the symbol
$\cL(\cU,\cY)$ stands for the space of bounded linear operators mapping $\cU$
into $\cY$, abbreviated to $\cL(\cU)$ in case $\cU=\cY$.
Following the standard terminology, we define the operator-valued {\em Schur class}
$\cS(\cU, \cY)$ to be the class of analytic functions $S$ on the open unit disk
$\D$ whose values $S(z)$  are contraction operators in $\cL(\cU,\cY)$.
By $H^2_\cU$ we denote the standard Hardy space of analytic
$\cU$-valued functions
on $\D$ with square-summable sequence of Taylor coefficients.
We also make use of the notation $\operatorname{Hol}_{\cU}({\mathbb
D})$ for the space of all $\cU$-valued holomorphic functions on the
unit disk ${\mathbb D}$.

Among several alternative characterizations of the Schur class there is
one in terms of positive kernels and associated reproducing kernel
Hilbert spaces:  {\em A function $S \colon {\mathbb D} \to
\cL(\cU,\cY)$ is in the Schur class $\cS(\cU,\cY)$ if and only if the
associated de Branges-Rovnyak kernel
\begin{equation}
K_S(z,\zeta)=\frac{I_\cY-S(z)S(\zeta)^*}{1-z\bar{\zeta}}
\label{1.1}
\end{equation}
is positive} (precise definitions are recalled at the end of this
Introduction). This positive kernel gives rise to a reproducing kernel Hilbert
space ${\mathcal H}(K_S)$, the  de Branges-Rovnyak space defined by $S$
(see \cite{dbr2}).
On the other hand, the kernel \eqref{1.1} being positive is equivalent to
the operator $M_S: f\to Sf$ of  multiplication by $S$ being a contraction
in $\cL(H^2_\cU,H^2_\cY)$; then the  general complementation theory applied
to the contractive operator $M_S \colon H_\cU^2\to H_\cY^2$ provides
the characterization
of $\cH(K_S)$ as the operator range
$\cH(K_S)={\rm Ran}(I-M_SM^*_S)^{\frac{1}{2}}\subset H^2_\cY$  with the
lifted norm
$$
\|(I-M_SM^*_S)^{\frac{1}{2}}f\|_{\cH(K_S)}=\|(I-{\bf p})f \|_{H^2_\cY}
$$
where ${\bf p}$ here is the orthogonal projection onto ${\rm
Ker}(I-M_SM^*_S)^{\frac{1}{2}}$.
Upon setting $f=(I-M_SM^*_S)^{\frac{1}{2}}h$ in the last formula we get
\begin{equation}
\| (I - M_{S}M_{S}^{*})h\|_{\cH(K_S)}=\langle (I - M_S M_S^{*})h, \,
h\rangle_{H^2_\cY}.
\label{1.2}
\end{equation}

The data set of the problem $\AP$ is a tuple
\begin{equation}
{\mathcal D}=\{S, T, E, N, {\bf x}\}
\label{1.2a}
\end{equation}
consisting of a Schur-class function $S\in\cS(\cU,\cY)$, Hilbert space
operators $T\in\cL(\cX)$, $E\in\cL(\cX,\cY)$, $N\in\cL(\cX,\cU)$, and
a vector ${\bf x}\in\cX$. With this data set we associate the
observability operators
\begin{equation}
\cO_{E,T}: \; x\mapsto E(I-zT)^{-1}x\quad\mbox{and}\quad \cO_{N,T}:
\; x\mapsto N(I-zT)^{-1}x,
\label{1.4a}
\end{equation}
which we assume  map $\cX$ into $\operatorname{Hol}_{\cY}({\mathbb D})$ and
$\operatorname{Hol}_{\cU}({\mathbb D})$.  We also associate with the
data set ${\mathcal D}$ the $\cL(\cX,\cY)$-valued function
\begin{equation}
F^S(z)=(E-S(z)N)(I-zT)^{-1}
\label{1.17}
\end{equation}
along with the multiplication operator $M_{F^S} \colon x\to F^Sx$, mapping $\cX$ into
$\operatorname{Hol}_{\cY}({\mathbb D})$.  Using the notation \eqref{1.4a}, we can
write $M_{F^{S}}$ as
\begin{equation}
M_{F^S}=\cO_{E,T}-M_S\cO_{N,T}:\cX\to\operatorname{Hol}_{\cY}({\mathbb D}).
\label{1.18a}
\end{equation}
Observe that, for an operator $A \colon \cX\to \cH(K_S)\subset H^2_\cY$, the adjoint
operator can be taken in the metric of $H^2_\cY$ as well as in the metric
of $\cH(K_S)$ which are not the same unless $S$ is inner (i.e., the multiplication
operator $M_S: \; H^2_\cU\to H^2_\cY$ is an isometry). To avoid confusion,
in what follows we use the notation $A^*$ for the adjoint of  $A$ in the metric
of $H^2_\cY$ and $A^{[*]} $ for the adjoint of $A$ in the metric
of $\cH(K_S)$.

\begin{defn}\label{D:1.1}
We say that the data set \eqref{1.2a} is {\em $\AP$-admissible} if:
\begin{enumerate}
\item The operators $\cO_{E,T}$ and $\cO_{N,T}$ map $\cX$ into
${\rm Hol}_\cY(\D)$ and ${\rm Hol}_\cU(\D)$, respectively (in other words,
$(E,T)$ and $(N,T)$ are {\em analytic output pairs}).
\item The operator $M_{F^S}$ maps $\cX$ into $\cH(K_S)$.
\item The operator $P:=M_{F^S}^{[*]}M_{F^S}$ satisfies the Stein equation
\begin{equation}
P-T^*PT=E^*E-N^*N.
\label{1.11}
\end{equation}
\end{enumerate}
\end{defn}

We are now ready to formulate the problem $\AP$:

\medskip

{\em Given an $\AP$-admissible
data set \eqref{1.2a}, find all $f\in\cH(K_S)$ such that
\begin{equation}
M_{F^S}^{[*]}f={\bf x}\quad\mbox{and}\quad\|f\|_{\cH(K_S)}\le 1.
\label{1.3a}
\end{equation}}

The $\AP$-problem as formulated here does not appear to be an
interpolation problem, but in Section \ref{S:OAP} we show that indeed
the operator-argument Nevanlinna-Pick interpolation problem can be seen as a
particular instance of the $\AP$-problem.

This operator-argument problem was considered in \cite{bbt2} for scalar-valued functions and for the nondegenerate case where
the solution $P$ of the Stein equation \eqref{1.11} is positive definite (i.e., invertible).
The eventual parametrization for the set of all solutions, which we obtain in Theorem
\ref{T:AIPsol} below, is connected with previously appearing representations for almost
invariant subspaces and Toeplitz kernels in terms of an isometric multiplier between
two de Branges-Rovnyak spaces.
As another application of the $\AP$-problem, we obtain an alternative
characterization of Toeplitz kernels (in Corollary
\ref{C:Toeplitzkernel} below) in terms of an explicitly computable isometric multiplier on an appropriate
de Branges-Rovnyak space; this is a refinement of the characterization due to Dyakonov \cite{dy}.

At one level the interpolation problem $\AP$ is straightforward since de
Branges-Rovnyak spaces are Hilbert spaces and consequently the set of all
norm-constrained solutions splits as the orthogonal direct sum of the unique
minimal-norm solution and the set of all functions satisfying the homogeneous
interpolation condition and the complementary norm constraint.
By viewing \eqref{1.3a} as a special case of a basic linear operator
equation discussed in Section \ref{S:LOE}, we get some general results
on the $\AP$-problem in Section \ref{S:FirstResults}.  These results
make no use of condition (3) (i.e., the Stein equation \eqref{1.11})
in the definition of $\AP$-admissibility, and can be easily extended
to a more general framework of contractive multipliers between any two
reproducing kernel Hilbert spaces (not necessarily of de
Branges-Rovnyak type).
By using the full strength of $\AP$-admissibility, in Section \ref{S:Param} we obtain
a more explicit formula (see Theorem \ref{T:AIPsol} below) for the parametrization of
the solution set by using the connection with an associated Schur-class Abstract Interpolation
Problem and its known Redheffer transform solution as worked out in
\cite{kky}.
The latter problem and its solution through the associated Redheffer transform is recalled
in Section \ref{S:Redheffer}. This section also includes an analysis of the conditions
under which the Redheffer transform is injective, a property which does not happen in general.
The paper concludes with three sections that discuss the various
applications of the $\AP$-problem mentioned above.

The notation is mostly standard. We just mention that an operator $X\in\cL(\cY)$, for
some Hilbert space $\cY$, is called {\em positive semidefinite} in case $\inn{Xy}{y}\geq 0$
for all $y\in\cY$ and {\em positive definite} if $X$ is positive semidefinite and invertible in
$\cL(\cX)$.   Also, in general, given a function $K$ defined on a Cartesian product set
$\Omega \times \Omega$ with values in $\cL(\cY)$, we say that $K$ is
a {\em positive kernel} if any one of the following equivalent
conditions hold:
\begin{enumerate}
    \item {\em $K$ is a  positive kernel in the sense of Aronszajn}:  given any finite
    collection of points $\omega_{1}, \dots, \omega_{N}$ in $\Omega$
    and vectors $y_{1}, \dots, y_{N}$ in the Hilbert coefficient
    space $\cY$, it holds that
 $$
 \sum_{i,j=1}^{N} \langle (K(\omega_{i}, \omega_{j}) y_{j}, y_{i}
 \rangle_{\cY} \ge 0.
 $$

 \item {\em  $K$ is the reproducing kernel for a reproducing kernel
 Hilbert space $\cH$}:  there is a Hilbert space $\cH(K)$ whose
 elements are $\cY$-valued functions on $\Omega$ so that (i) for each
 $\omega \in \Omega$ and $y \in \cY$ the $\cY$-valued function
 $k_{\omega}y$ given by $k_{\omega}y(\omega') = K(\omega', \omega) y$
 is an element of $\cH(K)$, and (ii)  the functions $k_{\omega}y$
 have the reproducing property for $\cH(K)$:
 $$
  \langle f, \, k_{\omega}y \rangle_{\cH(K)} = \langle f(\omega), y
  \rangle_{\cY}
 $$
 for all $f \in \cH(K)$.

 \item {\em $K$ has a Kolmogorov decomposition}:  there is an auxiliary
 Hilbert space ${\mathcal K}$ and a function $H \colon \Omega \to
 \cL(\cK, \cY)$ so that $K$ can be expressed as
 $$
  K(\omega', \omega) = H(\omega') H(\omega)^{*}.
$$

\end{enumerate}
These equivalences are well-known straightforward extensions of the
ideas of Aronszajn \cite{Aron} to the case of operator-valued kernels
in place of scalar-valued kernels.

Next we mention that  on occasion we view a vector ${\mathbf x}$ in a Hilbert space $\cX$ as an operator from the scalars ${\mathbb C}$ into $\cX$:  ${\mathbf x}$ maps the scalar $c \in {\mathbb C}$
to the vector $c {\mathbf x} \in {\mathcal X}$.
Then ${\mathbf x}^*$ denotes the adjoint operator mapping
${\mathcal X}$ back to ${\mathbb C}$:
${\mathbf x}^*({\mathbf y}) = \langle {\mathbf  y}, {\mathbf  x} \rangle \in {\mathbb C}$.
We will use the notation ${\mathbf x}^*$ for this operator rather than the more cumbersome $\langle \cdot, {\mathbf x} \rangle$.

Finally we note that a crucial tool for many of the results of this paper is the manipulation of $2 \times 2$ block matrices centering around the so-called Schur complement.  Given any $2 \times 2$ block matrix $M =  \left[ \begin{smallmatrix} A & B \\ C & D \end{smallmatrix} \right]$ with $A$ invertible, we define the {\em Schur complement} of $D$ (with respect to $M$) to be the matrix
$$
S_M(D):= D -  C A^{-1}B.
$$
In case $D$ is invertible, we define the {\em Schur complement}  of $A$ (with respect to $M$) to be the matrix
$$
  S_M(A): = A  - B D^{-1} C.
$$
Our main application is to the case where $M=M^*$ is self-adjoint (so $A=A^*$, $D=D^*$ and $C = B^*$).
Assuming $A$ is invertible, we may factor $A$ as $A = |A|^{1/2 } J |A|^{1/2}$ where $J: = \text{sign}(A)$ and the factorization
$$
\begin{bmatrix} A & B \\ B^* & D \end{bmatrix}
= \begin{bmatrix} |A|^{1/2} & 0 \\ B^* |A|^{-1/2}J & I \end{bmatrix}
\begin{bmatrix} J & 0 \\ 0 & D - B^* A^{-1} B \end{bmatrix}
\begin{bmatrix} |A|^{1/2} & J |A|^{-1/2} B \\ 0 & I \end{bmatrix}
$$
shows that {\em $M \ge 0$ (i.e., $M$ is positive-semidefinite) if and only if $A \ge 0$ (so $J = I$) and the Schur complement $D - B^* A^{-1} B$ of $D$ is positive semidefinite. }  Similarly, in case $D$ is invertible,  we see that {\em  $M \ge 0$ if and only if $D \ge 0$ and the Schur complement of $A$, namely, $A - B D^{-1} B^*$, is positive semidefinite. In fact, these results go through without the invertibility assumption on
$A$ or $D$, using Moore-Penrose inverses instead.}

\section{Linear operator equations}\label{S:LOE}

The problem $\AP$ is a particular case of the following well-known
norm constrained operator problem: {\em Given
$A\in\cL(\cH_2,\cH_3)$ and $B\in\cL(\cH_1,\cH_3)$, with
$\cH_1$, $\cH_2$ and $\cH_3$ given Hilbert spaces, describe
the operators $X\in\cL(\cH_1,\cH_2)$ that satisfy}
\begin{equation}
\label{2.1}
AX = B \quad\text{\em and}\quad \|X\| \le 1.
\end{equation}
The solvability criterion is known as
the Douglas factorization lemma \cite{Douglas}.

\begin{lem}\label{L:Douglas}
There exists an $X\in\cL(\cH_1,\cH_2)$ satisfying \eqref{2.1} if and only
if $AA^*\ge BB^*$. In this case, there exists a unique
$X\in\cL(\cH_1,\cH_2)$
satisfying \eqref{2.1} and the additional constraints
${\rm Ran}X\subset{\rm Ran}A^*$
and ${\rm Ker}X={\rm Ker}B$.
\end{lem}

In case $A A^{*}\geq BB^{*}$, Lemma \ref{L:Douglas} guarantees the
existence of (unique) contractions $X_1\in\cL(\cH_1,\overline{{\rm
Ran}}A)$ and
$X_2\in\cL(\cH_2,\overline{\rm Ran}A)$ so that
\begin{equation}\label{2.2}
(AA^*)^{\frac{1}{2}}X_1=B,\quad
(AA^*)^{\frac{1}{2}}X_2=A, \quad {\rm Ker} X_1={\rm Ker} B,\quad
{\rm Ker} X_2={\rm Ker} A.
\end{equation}
By construction, $X_2$ is a coisometry. The next lemma gives a
description of the operators $X\in\cL(\cH_1,\cH_2)$ satisfying \eqref{2.1}
in terms of the operators $X_1$ and $X_2$.

\begin{lem}
Assume $A A^{*}\geq BB^{*}$ and let $X\in\cL(\cH_1,\cH_2)$.
Then the following statements are equivalent:
\begin{enumerate}
\item $X$ satisfies conditions \eqref{2.1}.

\item The operator
\begin{equation}\label{2.3}
\begin{bmatrix}  I_{\cH_{1}} & B^{*} & X^{*} \\
B & A A^{*} & A \\ X & A^{*} & I_{\cH_{2}} \end{bmatrix}: \;
\begin{bmatrix}\cH_1\\\cH_3\\\cH_2\end{bmatrix}\to
\begin{bmatrix}\cH_1\\\cH_3\\\cH_2\end{bmatrix}
\end{equation}
is positive semidefinite.

\item $X$ is of the form
\begin{equation}\label{2.4}
X=X_2^*X_1+(I-X_2^*X_2)^{\frac{1}{2}}K (I-X_1^*X_1)^{\frac{1}{2}}
\end{equation}
where $X_1$ and $X_2$ are defined as in \eqref{2.2} and
where the parameter $K$ is an arbitrary contraction from
$\overline{\rm Ran}(I-X_1^*X_1)$ into $\overline{\rm Ran}(I-X_2^*X_2)$.
\end{enumerate}
Moreover, if $X$ satisfies \eqref{2.1}, then $X$ is unique if and only if
$X_1$ is isometric on $\cH_1$ or $X_2$ is isometric on $\cH_2$.
\label{L:2.1}
\end{lem}

\begin{proof}
Note that positivity of the block-matrix in \eqref{2.3} is equivalent to
positivity of the Schur complement of $I_{\cH_2}$ in \eqref{2.3}, namely
\begin{equation}\label{2.5}
\begin{bmatrix} I & B^{*} \\ B & A A^{*} \end{bmatrix}-
\begin{bmatrix} X^{*} \\ A \end{bmatrix}
\begin{bmatrix} X & A^* \end{bmatrix} = \begin{bmatrix}  I -
X^{*}X & B^{*} - X^{*} A^{*} \\ B - A X & 0 \end{bmatrix}\ge 0.
\end{equation}
Because of the zero in the $(2,2)$-entry of the left hand side of
\eqref{2.5},
we find that the inequality \eqref{2.5} holds precisely when
$$
B-AX=0\quad\mbox{and}\quad I-X^*X\geq 0,
$$
which is equivalent to \eqref{2.1}. On the other hand, condition
\eqref{2.3} is equivalent, by taking the Schur complement of $AA^*$
in \eqref{2.3} and making use of \eqref{2.2}, to
$$
\begin{bmatrix} I & X^* \\ X& I\end{bmatrix}-
\begin{bmatrix}X_1^* \\ X_2^*\end{bmatrix}\begin{bmatrix}X_1&
X_2\end{bmatrix}=\begin{bmatrix} I-X_1^*X_1 & X^*-X_1^*X_2 \\ X-X_2^*X_1 &
I-X_2^*X_2\end{bmatrix}\ge 0.
$$
By Theorem XVI.1.1 from \cite{FF90}, the latter inequality is equivalent
to the representation \eqref{2.4} for $X$ with $K$ some contraction in
$\cL(\overline{\rm Ran}(I-X_1^*X_1),\overline{\rm Ran}(I-X_2^*X_2))$.
Moreover, $X$ and $K$ in \eqref{2.4} determine each other uniquely.
The last statement in the lemma now follows from representation
\eqref{2.4}.
\end{proof}

Note that since $X_2$ is a coisometry, it follows that
$(I-X_2^*X_2)^{\frac{1}{2}}$ is the
orthogonal projection onto $\cH_1\ominus{\rm Ker} A=\cH_1\ominus{\rm Ker}
X_1$. This implies that for each $K$ in
\eqref{2.4} and  each $h\in\cH_1$, we have
$$
\|Xh\|^2=\|X_2^*X_1 h\|^2+\|(I-X_2^*X_2)^{\frac{1}{2}}K
(I-X_1^*X_1)^{\frac{1}{2}}h\|^2,
$$
so that $X_2^*X_1$ is the minimal norm solution to the problem
\eqref{2.1}.

\section{The $\AP$-problem as a linear operator equation}\label{S:FirstResults}

In this section we consider data sets ${\mathcal D}=\{S, T, E, N, {\bf x}\}$
satisfying conditions (1) and (2) in the definition of an
$\AP$-admissible data set but not necessarily condition (3); condition
(3) of an $\AP$-admissible data set (i.e., the Stein equation
\eqref{1.11}) comes to the fore for the derivation of the more
explicit results to be presented in Section  \ref{S:Param}.  We still speak
of the $\AP$-problem for this looser notion of admissible data
set.   Define $F^S$ as in \eqref{1.17}. If we apply Lemma \ref{L:Douglas} to the case where
\begin{equation}
A=M_{F^S}^{[*]}:\cH(K_S)\to\cX,\quad B ={\bf x}\in\cX \cong\cL(\C,\cX),
\label{2.7}
\end{equation}
then we see that solutions $X \colon \C\to\cH(K_{S})$ to problem \eqref{2.1} necessarily
have the form of a multiplication operator $M_f$ for some function
$f\in \cH(K_S)$.  This observation leads to the following solvability criterion.

\begin{thm}\label{T:3.6a}
The problem $\AP$ has a solution if and only if
\begin{equation}
P\ge {\bf x}{\bf x}^*,\quad\mbox{where}\quad
P:=M_{F^S}^{[*]}M_{F^S}.
\label{2.6}
\end{equation}
\end{thm}

\begin{rem}
{\rm Observe that for the unconstrained version of the problem $\AP$, the
existence criterion follows immediately from the definition \eqref{2.6} of
$P$:} there is a function $f\in\cH(K_S)$ such that $M^{[*]}_{F^S}f={\bf
x}$ if and only if ${\bf x}\in{\rm Ran}P^{\frac{1}{2}}$.
\label{R:2.30}
\end{rem}

The next theorem characterizes
solutions to the problem $\AP$ in terms of a positive kernel.
 We emphasize that
characterizations of this type
go back to the Potapov's method of fundamental matrix inequalities
\cite{kopot}.  Here the notation ${\mathbf x}^*$ associated with a vector ${\mathbf x} \in {\mathcal X}$ follows the conventions explained at the end of the Introduction.

\begin{thm}
A function $f: \, \D\to \cY$ is a solution of the problem $\AP$ with data set \eqref{1.2a}
if and only if the kernel
\begin{equation}
{\bf K}(z,\zeta)=\begin{bmatrix} 1 & {\bf x}^* & f(\zeta)^*\\
{\bf x} & P & F^S(\zeta)^{*} \\ f(z) & F^S(z) & K_S(z,\zeta)
\end{bmatrix}\qquad (z,\zeta\in\D),
\label{2.8}
\end{equation}
is positive on $\D\times\D$. Here $P$, $F^S$ and $K_S$ are given by
\eqref{2.6}, \eqref{1.17} and \eqref{1.1}, respectively.
\label{T:2.1}
\end{thm}

\begin{proof}
By Lemma \ref{L:2.1} specialized to $A$ and $B$ as in \eqref{2.7}
and $X=M_f$, we conclude that $f$ is a solution to the problem $\AP$
(that is, it meets conditions \eqref{1.3a}) if and only if the following
operator is positive semidefinite:
$$
{\bf P}:=\begin{bmatrix}1 & {\bf x}^* & M_f^{[*]}
\\ {\bf x} & M_{F^S}^{[*]}M_{F^S} &  M_{F^S}^{[*]} \\ M_f &  M_{F^S}
& I_{\cH(K_S)}
\end{bmatrix}= \begin{bmatrix}1 & {\bf x}^* & M_f^{[*]} \\
{\bf x} & P &  M_{F^S}^{[*]} \\ M_f &  M_{F^S} & I_{\cH(K_S)}
\end{bmatrix}\ge 0.
$$
We next observe that for every vector $g\in\C\oplus{\mathcal X}\oplus
\cH(K_S)$
of the form
\begin{equation}
g(z)=\sum_{j=1}^r\left[\begin{array}{c}c_j \\ x_j\\
K_S(\cdot,z_j)y_j\end{array}\right]\qquad
(c_j\in\C, \; y_j\in\cY, \; x_j\in{\mathcal X}, \; z_j\in\D)
\label{2.9}
\end{equation}
the identity
\begin{equation}
\left\langle {\bf P}g, \; g\right\rangle_{\C\oplus\cX\oplus
\cH(K_S)}=\sum_{j,\ell=1}^r \left\langle{\bf K}(z_j,z_\ell)
\sbm{c_\ell \\ x_\ell \\ y_\ell}, \; \sbm{c_j \\ x_j \\ y_j}
\right\rangle_{\C\oplus{\mathcal X}\oplus \cY}
\label{2.10}
\end{equation}
holds. Since the set of
vectors of the form \eqref{2.9} is dense in $\C\oplus{\mathcal X}\oplus
\cH(K_S)$, the identity \eqref{2.10} now implies that the operator
${\bf P}$ is positive semidefinite if and only if the quadratic form on the
right hand side of \eqref{2.10} is nonnegative, i.e., if and only if
the kernel \eqref{2.8} is positive on $\D\times\D$.
\end{proof}

For the rest of this section we assume that the operator $P$ in \eqref{2.6}
is positive definite. Then the operator $M_{F^S}P^{-\frac{1}{2}}$ is an
isometry and the space
\begin{equation}
{\mathcal N}=\{F^S(z)x: \, x\in\cX\}
\quad\mbox{with norm}\quad \|F^Sx\|_{\cH(S)}=\|P^{\frac{1}{2}}x\|_{_\cX}
\label{2.11}
\end{equation}
is isometrically included in  $\cH(K_S)$. Moreover, the orthogonal complement of ${\mathcal N}$
in $\cH(K_S)$ is the reproducing kernel Hilbert space $\cH(\widetilde{K}_S)$
with reproducing kernel
\begin{equation}
\widetilde{K}_S(z,\zeta)=
K_S(z,\zeta)-F^S(z)P^{-1}F^S(\zeta)^{*}.
\label{2.12}
\end{equation}

\begin{thm}\label{T:3.4}
Assume that condition \eqref{2.6} holds and that $P$ is positive definite.
Let $\widetilde{K}_S(z,\zeta)$ be the kernel defined in \eqref{2.12}. Then:
\begin{enumerate}
\item All solutions $f$ to the problem $\AP$ are described by the
formula
\begin{equation}
f(z)=F^S(z)P^{-1}{\bf x}+h(z)
\label{2.14}
\end{equation}
where $h$ is a free parameter from $\cH(\widetilde{K}_S)$  subject to
$$
\|h\|_{\cH(\widetilde{K}_S)}\le \sqrt{1-\|P^{-\frac{1}{2}}{\bf
x}\|^2}.
$$
\item The problem $\AP$ has a unique solution if and only if
\begin{equation}
\|P^{-\frac{1}{2}}{\bf x}\|=1\quad\mbox{or}\quad
\widetilde{K}_S(z,\zeta)\equiv 0.
\label{2.13}
\end{equation}
\end{enumerate}
\end{thm}

\begin{proof}
It is readily seen that
\[
X_1=P^{-\frac{1}{2}}{\bf x}\in\cX \cong\cL(\BC,\cX),\quad
X_2=P^{-\frac{1}{2}}M_{\widetilde{F}^S}^{[*]}\in\cL(\cH(K_S),\cX)
\]
are the operators $X_1$ and $X_2$ from \eqref{2.2} after specialization to the case \eqref{2.7}.


The second statement now follows from Lemma \ref{L:2.1}, since
$P^{-\frac{1}{2}}{\bf x}\in\cL(\C,\cX)$ being isometric means that
$\|P^{-\frac{1}{2}}{\bf x}\|=1$
and, on the other hand, the isometric property for the operator
$M_{F^S}P^{-\frac{1}{2}}$ means that the space ${\mathcal N}$
defined in \eqref{2.11} is equal to the
whole space $\cH(K_S)$.  Thus $\cH(\widetilde K_{S}) = \cH(K_{S})
\ominus \cN = \{0\}$ or $\widetilde K_{S} \equiv 0$.

In the present framework, the parametrization formula \eqref{2.4} takes
the form
\begin{equation}
M_f=M_{F^S}P^{-\frac{1}{2}}{\bf x}+\sqrt{1-\|P^{-\frac{1}{2}}{\bf
x}\|^2}\cdot
(I-M_{F^S}P^{-1}M_{F^S}^{[*]})^{\frac{1}{2}}K
\label{2.16}
\end{equation}
where $K$ is equal to the operator of multiplication $M_k: \,
\C\to\cH(\widetilde{K}_S)$ by a function $k\in\cH(\widetilde{K}_S)$
with $\|k\|\le 1$. Since $M_{F^S}P^{-\frac{1}{2}}$ is an isometry,
the second term on the right hand side of \eqref{2.16} is equal to
the operator $M_h$ of multiplication by a function $h\in
\cH(\widetilde{K}_S)$ such that $\|h\|_{\cH(\widetilde{K}_S)}
=\|h\|_{\cH(K_S)}\le \sqrt{1-\|P^{-\frac{1}{2}}{\bf x}\|^2}$.
\end{proof}

\begin{rem}
{\rm The second term $h$ on the right hand side of \eqref{2.14} represents
in fact the general solution of the homogeneous interpolation problem
(with interpolation condition $M_{F^S}^{[*]}f=0$). If $h$ runs through
the whole space $\cH(\widetilde{K}_S)$, then formula \eqref{2.14}
produces all functions $f\in\cH(K_S)$ such that $M_{F^S}^{[*]}f={\bf x}$.
This unconstrained interpolation problem has a unique solution
if and only if $\widetilde{K}_S(z,\zeta)\equiv 0$. Thus, the second
condition in \eqref{2.13} provides the uniqueness of an $f$ subject to
$M_{F^S}^{[*]}f={\bf x}$ in the whole space $\cH(K_S)$, not just in the unit
ball of  $\cH(K_S)$. If $\widetilde{K}_S(z,\zeta)\not\equiv 0$, then the
unconstrained problem
has infinitely many solutions and, as in the general framework, the
function $F^S(z)P^{-1}{\bf x}$ has the minimal possible
norm. Since $M_{F^S}P^{-\frac{1}{2}}$ is an isometry,
it follows from \eqref{2.11} that
$\|M_{F^S}P^{-1}{\bf x}\|_{\cH(K_S)}=\|P^{-\frac{1}{2}}{\bf x}\|$. Thus,
if $\|P^{-\frac{1}{2}}{\bf x}\|=1$, then uniqueness occurs
since the {\em minimal norm solution} already has unit norm.}
\label{R:2.6}
\end{rem}

\section{Redheffer transform related to the AIP-problem on $\cS(\cU,\cY)$}
\label{S:Redheffer}

To obtain a more explicit parametrization of the solution set
to the $\AP$-problem, we need some facts concerning the Abstract
Interpolation Problem for functions in the Schur Class
$\cS(\cU, \cY)$ (denoted as the $\APS$-problem) from \cite{kky} (see also \cite{khy, khberk})
which we now recall.

We consider the  data set
\begin{equation}
{\mathcal D}^\prime=\{P,T,E,N\}
\label{3.1}
\end{equation}
consisting of operators $P,T\in\cL(\cX)$, $E\in\cL(\cX,\cY)$,
$N\in\cL(\cX,\cU)$ such that the pairs $(E,T)$ and $(N,T)$ are output
analytic and $P$ is a positive semidefinite solution of the Stein
equation \eqref{1.11}. A set with these properties is called
{\em $\APS$-admissible}.\smallskip

$\APS$: {\em Given an $\APS$-admissible data set \eqref{3.1},
find all functions $S: \, \D\to\cL(\cU,\cY)$  such that the
kernel
\begin{equation}
(z,\zeta)\mapsto\begin{bmatrix}P & F^S(\zeta)^* \\ F^S(z) &
K_S(z,\zeta)\end{bmatrix}\quad (z,\zeta\in\D)
\label{3.2}
\end{equation}
is positive on $\D\times \D$, or equivalently, find all functions
$S\in\cS(\cU,\cY)$ so that the operator
$M_{F^S}=\cO_{E,T}-M_S\cO_{N,T}: x\to F^S(z)x$
maps $\cX$ into $\cH(K_S)$ and satisfies $M_{F^S}^{[*]}M_{F^S}\le P$.
Here $F^S$ is the function defined in \eqref{1.17}.}

\smallskip

The equivalence of the two above formulations follows from a general
result on reproducing kernel Hilbert spaces; see \cite{beabur}.

The parametrization of the solutions through an associated Redheffer
transform is recalled in Theorem \ref{T:3.1} below. To state
the result we need to construct the Redheffer transform.
Observe that \eqref{1.11} can equivalently be written as
\[
\|P^{\frac{1}{2}}x\|^2+\|Nx\|^2=\|P^{\frac{1}{2}}Tx\|^2+\|Ex\|^2\quad\text{for
all}\quad x\in\cX.
\]
Let $\cX_0=\ov{\Ran} P^{\frac{1}{2}}$. With some abuse of notation we will occasionally
view $P$ as an operator mapping $\cX$ into $\cX_0$, or $\cX_0$ into $\cX$, while
still using $P=P^*$.
The above identity shows that there exists a well defined isometry
$V$ with domain ${\mathcal D}_V$ and range ${\mathcal R}_V$ equal
to
$$
    {\mathcal {D}}_V=\overline{\rm Ran}\left[\begin{array}{c}
    P^\half \\ N\end{array}\right]\subseteq \left[\begin{array}{c}
    \cX_0 \\  \cU\end{array}\right]\quad\mbox{and}\quad
    {\mathcal {R}}_V=\overline{\rm Ran}
    \left[\begin{array}{c}P^\half T\\
    E\end{array}\right]\subseteq \left[\begin{array}{c}
    \cX_0 \\  \cY\end{array}\right],
$$
respectively, which is uniquely determined by the identity
 \begin{equation}\label{3.3}
    V \left[\begin{array}{c}P^\half x \\
    Nx\end{array}\right]
   = \left[\begin{array}{c} P^\half Tx\\
    Ex\end{array}\right]\quad\mbox{for all} \; \; x\in\cX.
    \end{equation}
We then define the defect spaces
    \begin{equation}\label{3.4}
    \Delta:=\left[\begin{array}{c}\cX_0 \\ \cU\end{array}\right]
    \ominus{\mathcal D}_V \quad{\rm and}\quad
    \Delta_*:=\left[\begin{array}{l}\cX_0\\
    \cY\end{array}\right]\ominus{\mathcal {R}}_V,
    \end{equation}
    and let $\widetilde{\Delta}$ and $\widetilde{\Delta}_*$ denote
    isomorphic copies of $\Delta$ and $\Delta_*$, respectively, with
    unitary identification maps
    $$
    {\bf i}: \; \Delta\rightarrow \widetilde{\Delta}\quad\mbox{and}\quad
    {\bf i}_*: \; \Delta_*\rightarrow \widetilde{\Delta}_*.
    $$
With these identification maps we define a unitary colligation matrix {\bf U} from
$\cD_V\oplus\De\oplus\wtil\De_*=\cX\oplus\cU\oplus\wtil\De_*$
onto $\cR_V\oplus\De_*\oplus\wtil\De=\cX\oplus\cY\oplus\wtil\De$
by
\begin{equation}\label{AIPcol1}
{\bf U}=\mat{ccc}{V & 0 & 0\\0 & 0 & {\mathbf i}_*^*\\0 & {\mathbf i} & 0}
:\mat{c}{\cD_V\\\De\\\wtil\De_*}\to\mat{c}{\cR_V\\\De_*\\\wtil\De},
\end{equation}
which we also decompose as
\begin{equation}\label{3.6}
{\bf U}=\mat{ccc}{A&B_1&B_2\\C_1&D_{11}&D_{12}\\C_2&D_{21}&0}
:\mat{c}{\cX_0\\\cU\\\wtil\De_*}\to\mat{c}{\cX_0\\\cY\\\wtil\De}.
\end{equation}
Write $\Si$ for the characteristic function associated with this
colligation ${\bf U}$, i.e.,
\begin{equation}\label{3.7}
\Si(z)=\mat{cc}{D_{11}&D_{12}\\D_{21}&0}+z\mat{c}{C_1\\C_2}(I-zA)^{-1}\mat{cc}{B_1&B_2}\quad
(z\in\D),
\end{equation}
and decompose $\Si$ as
\begin{equation}\label{RedCoefs2}
\Si(z)=\mat{cc}{\Si_{11}(z)&\Si_{12}(z)\\\Si_{21}(z)&\Si_{22}(z)}
:\mat{c}{\cU\\\wtil{\De}_*}\to\mat{c}{\cY\\\wtil{\De}}.
\end{equation}
A straightforward calculation based on the fact that ${\bf U}$ is coisometric
gives
\begin{equation}
\frac{I-\Sigma(z)\Sigma(\zeta)^*}{1-z\overline{\zeta}}=
\begin{bmatrix}C_1\\C_2\end{bmatrix}(I-zA)^{-1}(I-\overline{\zeta}A^*)^{-1}
\begin{bmatrix}C_1^*&C_2^*\end{bmatrix},
\label{3.8}
\end{equation}
which implies in particular that $\Sigma$ belongs to the Schur class
$\cS(\cU\oplus\widetilde{\Delta}_*,\cY\oplus\widetilde{\Delta})$. Moreover,
it follows from the construction that $\Si_{22}(0)=0$. These facts
imply that the Redheffer linear fractional transform
\begin{equation}
S={\boldsymbol{\mathcal R}}_\Sigma[\cE]:=\Sigma_{11}+\Sigma_{12}(I-\cE\Sigma_{22})^{-1}\cE\Sigma_{21}
\label{3.9}
\end{equation}
is well defined for every Schur-class function
$\cE\in\cS(\widetilde{\Delta},\widetilde{\Delta}_*)$. The next theorem
(see \cite{kky} for the proof) shows that the image of the class
$\cS(\widetilde{\Delta},\widetilde{\Delta}_*)$ under the
Redheffer transform ${\boldsymbol{\mathcal R}}_\Sigma$ is precisely the
solution set of the problem $\APS$.

\begin{thm}
\label{T:3.1}
Given an $\APS$-admissible data set \eqref{3.1}, let
${\boldsymbol{\mathcal R}}_\Sigma$ be the  Redheffer transform
constructed as in \eqref{3.7}, \eqref{3.9}. A function $S: \,
\D\to\cL(\cU,\cY)$ is a solution of the problem $\APS$
if and only if $S={\boldsymbol{\mathcal R}}_\Sigma[\cE]$ for
some $\cE\in\cS(\widetilde{\Delta},\widetilde{\Delta}_*)$.
\end{thm}

The function $\Sigma_{11}$ appears as a solution upon taking $\cE\equiv 0$, and
is called {\em the central solution} of the problem $\APS$. In case the
problem has only one solution, this solution must be the central solution.

\begin{prop}\label{P:ContrMult}
Let $\Sigma$ be the characteristic function of the unitary colligation
${\bf U}$ in \eqref{3.6}, decomposed as in \eqref{RedCoefs2}, and let
$S={\boldsymbol{\mathcal R}}_\Sigma[\cE]\in\cS(\cU,\cY)$
for a given $\cE\in\cS(\widetilde{\Delta},\widetilde{\Delta}_*)$.
Define the functions
\begin{equation}\label{GGa}
\begin{array}{rl}
G(z)\!\!\!&=\Sigma_{12}(z)(I-\cE(z)\Sigma_{22}(z))^{-1},\\[.2cm]
\Gamma(z)\!\!\!&=\left(C_1+G(z)\cE(z)C_2\right)(I-z A)^{-1}.
\end{array}\qquad(z\in\BD)
\end{equation}
Then $G$ defines a contractive multiplier $M_G:\cH(K_\cE)\to\cH(K_S)$,
$\Ga$ a contractive multiplier $M_\Ga:\cX_0\to \cH(K_S)$, and the operator
\begin{equation}
\begin{bmatrix} M_G & M_\Gamma\end{bmatrix}: \; \begin{bmatrix}
\cH(K_\cE) \\ \cX_0\end{bmatrix}\to \cH(K_S)
\label{3.14}
\end{equation}
is coisometric. Furthermore,  we have
\begin{equation}  \label{3.14'}
    F^S(z)=\Ga(z)P^{\half}
\end{equation}
for each $z\in\D$,
where $F^S$ is the function defined in \eqref{1.17}. In particular,
$M_\Ga$ is an isometry and $M_G$ a partial isometry if and only if
$P=M_{F^{S}}^{[*]}M_{F^S}$ with $M_{F^S}:\cX_0\to\cH(K_S)$ defined by
$M_{F^S}=M_\Ga P^\half$.
\end{prop}

\begin{proof} The identity $\begin{bmatrix}I & G(z)\cE(z)\end{bmatrix}\Sigma(z)=
\begin{bmatrix}S(z) & G(z)\end{bmatrix}$ is an immediate consequence of
\eqref{3.9} and the definition of $G(z)$ in \eqref{GGa}. Using this identity one can
easily compute that
\begin{align*}
I-S(z)S(\zeta)^*=& G(z)(I-\cE(z)\cE(\zeta)^*)G(\zeta)^*+
\mat{cc}{I&G(z)\cE(z)}\mat{c}{I\\\cE(\zeta)^*G(\zeta)^*}\\
&\quad -\mat{cc}{S(z)&G(z)}\mat{c}{S(\zeta)^*\\G(\zeta)^*}\\
= &  G(z)\left(
I-{\cE}(z){\cE}(\zeta)^*\right)G(\zeta)^*  \\
&\quad +\begin{bmatrix}I &
G(z)\cE(z)\end{bmatrix}\left(I-\Sigma(z)\Sigma(\zeta)^*\right)
\begin{bmatrix}I \\ \cE(\zeta)^*G(\zeta)^*\end{bmatrix};
\end{align*}
(see also \cite[Lemma 8.3]{boldym}). Dividing both
sides of the latter identity by $1-z\overline{\zeta}$ leads to
$$
K_S(z,\zeta)= G(z)K_\cE(z,\zeta)G(\zeta)^*+\begin{bmatrix}I &
G(z)\cE(z)\end{bmatrix}K_\Sigma(z,\zeta)
\begin{bmatrix}I \\ \cE(\zeta)^*G(\zeta)^*\end{bmatrix}.
$$
By replacing $K_\Sigma(z,\zeta)$ by the expression on the right hand side of
\eqref{3.8}, we get
\begin{equation}\label{3.14b}
K_S(z,\zeta)= G(z)K_\cE(z,\zeta)G(\zeta)^*+\Gamma(z)\Gamma(\zeta)^*.
\end{equation}
It is easy to verify that
$$
    M_{G}^{*} \colon K_{S}(\cdot, \zeta) y \mapsto K_{S}(\cdot,
    \zeta) G(\zeta)^{*} y,  \quad M_{\Gamma}^{*} \colon K_{S}(\cdot,
    \zeta) y \mapsto \Gamma(\zeta)^{*}y
$$
from which one can deduce that the context of \eqref{3.14b} is that
$\begin{bmatrix} M_{G} & M_{\Gamma}^{*} \end{bmatrix}^{*}$ is
isometric on $\overline{\operatorname{span}} \{ K_{S}(\cdot, \zeta) y
\colon \zeta \in {\mathbb D}, y \in \cY\} = \cH(K_{S})$, i.e.,
$\begin{bmatrix} M_{G} & M_{\Gamma} \end{bmatrix}$ is coisometric as
asserted. In particular,
we see that $M_G$ and $M_\Ga$ are contractions.

To verify \eqref{3.14'},  recall
that the very construction of the colligation ${\bf U}$ implies that
$$
\begin{bmatrix}A&B_1\\C_1&D_{11}\\C_2&D_{21}\end{bmatrix}\begin{bmatrix}
P^\half\\ N\end{bmatrix}
=\begin{bmatrix}P^\half T\\E\\0\end{bmatrix},
$$
so that
$AP^\half-P^\half T=B_1N, \quad C_1P^\half=E-D_{11}N,\quad C_2P^\half=-D_{21}N$.
Making use of the latter equalities and of realization formulas
for $\Sigma_{11}$ and $\Sigma_{21}$ in \eqref{3.7} we
compute for $z\in\D$,
\begin{align}
\begin{bmatrix}C_1\\C_2\end{bmatrix}(I-z A)^{-1}P^\half(I-zT)
    &= \begin{bmatrix}C_1 \\C_2\end{bmatrix}P^\half
+z\begin{bmatrix}C_1\\C_2\end{bmatrix}(I-z
A)^{-1}(AP^\half-P^\half T)\notag \\  & =\begin{bmatrix}E-D_{11}N
\\-D_{21}N\end{bmatrix}-z\begin{bmatrix}C_1\\C_2\end{bmatrix}(I-z
A)^{-1}B_1 N\notag \\ &
=\begin{bmatrix}E\\0\end{bmatrix}-\begin{bmatrix}\Sigma_{11}(z)\\
\Sigma_{21}(z)\end{bmatrix}N.
\label{3.13}
\end{align}
Upon multiplying the left-hand side expression in \eqref{3.13} by
$\begin{bmatrix}I & G(z)\cE(z)\end{bmatrix}$ on the left and by
$(I-zT)^{-1}$ on the right, we get $\Gamma(z)P^\half$ by definition
\eqref{GGa}. Applying the same multiplications to the right hand
side expression gives, on account of \eqref{GGa} and \eqref{3.9},
$$
E(I-z T)^{-1}-(\Sigma_{11}(z)+G(z)\cE(z)\Sigma_{21}(z))N(I-z T)^{-1}
=(E-S(z)N)(I-z T)^{-1}
$$
which is $F^S(z)$. Thus $\Gamma(z)P^\half=F^S(z)$, and \eqref{3.14'}
follows.

It is now straightforward to verify that $M_\Ga$ is an isometry if and only if
$P=M_{F^{S}}^{[*]}M_{F^S}$, while, since \eqref{3.14} is a coisometry,
$M_\Ga$ being an isometry implies that $M_G$ is a partial isometry.
\end{proof}

\subsection{Injectivity of ${\boldsymbol R}_{\Sigma}$ and $M_{G}$}

In this subsection we focus on two questions: (1) when is the
above constructed Redheffer transform ${\boldsymbol{\mathcal R}}_\Sigma$
injective, and, (2) for a given $\cE\in\cS(\wtil{\De},\wtil{\De}_*)$, when is the
multiplication operator $M_G:\cH(K_\cE)\to\cH(K_S)$ from Proposition
\ref{P:ContrMult} injective (and thus an isometry if $P=M_{F^{S}}^{[*]}M_{F^S}$)?
The next lemma provides the basis for the results to follow.

\begin{lem}\label{L:noker}
Assume $M_{\Si_{12}}:\textup{Hol}_{\wtil{\De}_*}(\BD)\to\textup{Hol}_\cY(\BD)$ has a
trivial kernel and $M_{\Si_{21}}:\textup{Hol}_{\cU}(\BD)\to\textup{Hol}_{\wtil{\De}}(\BD)$
has dense range. Then the Redheffer transform ${\boldsymbol{\mathcal R}}_\Sigma$ is
injective, and for any $\cE\in\cS(\wtil{\De},\wtil{\De}_*)$ the multiplication operator
$M_G:\cH(K_\cE)\to\cH(K_S)$ has trivial kernel.
\end{lem}

\begin{proof}[\bf Proof]
The identity
\[
M_S=M_{\Si_{11}}+M_{\Si_{12}}(I-M_\cE M_{\Si_{22}})^{-1}M_\cE M_{\Si_{21}}
\]
may not hold if we consider the multiplication operators as acting between the appropriate
$H^2$-spaces, since $I-M_\cE M_{\Si_{11}}$ may not be boundedly invertible, but the identity
does hold when the multiplication operators are viewed as operators between the appropriate
linear spaces of holomorphic functions on $\BD$, i.e.,
\[
\begin{array}{c}
M_S:\textup{Hol}_{\wtil{\De}_*}(\BD)\to\textup{Hol}_\cY(\BD),\quad
M_\cE:\textup{Hol}_{\wtil{\De}_*}(\BD)\to\textup{Hol}_{\wtil{\De}}(\BD),\\[.2cm]
\mat{cc}{M_{\Si_{11}}&M_{\Si_{12}}\\M_{\Si_{21}}&M_{\Si_{22}}}:
\mat{c}{\textup{Hol}_{\cU}(\BD)\\\textup{Hol}_{\wtil{\De}_*}(\BD)}\to
\mat{c}{\textup{Hol}_{\cY}(\BD)\\\textup{Hol}_{\wtil{\De}}(\BD)}.
\end{array}
\]
Note that  $I-M_\cE M_{\Si_{22}}$ is invertible as a linear
map on $\operatorname{Hol}_{\widetilde \De_{*}}({\mathbb D})$
since $\Si_{22}(0)=0$ and ${\mathcal E}(z)$ and $\Sigma_{22}(z)$ are
both contractive for $z \in {\mathbb D}$.
Now assume
$\cE,\cE'\in\cS(\wtil{\De},\wtil{\De}_*)$ so that
${\boldsymbol{\mathcal R}}_\Sigma[\cE]={\boldsymbol{\mathcal R}}_\Sigma[\cE']$. By the
assumptions on $M_{\Si_{12}}$ and $M_{\Si_{21}}$ it follows that
$$
(I-M_{\cE}M_{\Si_{22}})^{-1}M_{\cE}=(I-M_{\cE'}M_{\Si_{22}})^{-1}M_{\cE'}
=M_{\cE'}(I-M_{\Si_{22}}M_{\cE'})^{-1},
$$
and thus
\begin{align*}
M_{\cE}-M_{\cE}M_{\Si_{22}}M_{\cE'}=&M_{\cE}(I-M_{\Si_{22}}M_{\cE'})\\
=&(I-M_{\cE}M_{\Si_{22}})M_{\cE'}=M_{\cE'}-M_{\cE}M_{\Si_{22}}M_{\cE'}.
\end{align*}
Hence $\cE=\cE'$. Since $M_{\Si_{12}}:\textup{Hol}_{\wtil{\De}_*}(\BD)\to\textup{Hol}_\cY(\BD)$
has a trivial kernel, so does $M_{\Si_{12}}(I-M_{\cE}M_{\Si_{22}})^{-1}$ when viewed as
an operator acting on $\textup{Hol}_{\wtil{\De}_*}(\BD)$, independently of the choice of
$\cE\in\cS(\wtil{\De},\wtil{\De}_*)$. In particular, $G(z)h(z)\equiv 0$ implies $h=0$ for any
$h\in\cH(K_\cE)$.
\end{proof}

The proof of the above lemma does not take into account the particularities of the
Redheffer transform associated with the problem $\APS$ constructed in
\eqref{3.3}--\eqref{3.6}, besides the fact that ${\Si_{22}}(0)=0$. As we shall see,
for the coefficients in the Redheffer transform we consider, $M_{\Si_{21}}$ always has
dense range, while $M_{\Si_{12}}$ has a trivial kernel if the operator $T^*$ is injective.
As preparation for this result, we need the following lemma.

\begin{lem}\label{L:kers}
Let $D_{21}$ and $D_{12}$ be the operators in the unitary colligation \eqref{3.6}.
Then
\begin{equation}
\ker D_{21}^*=\{0\}\quad\mbox{and}\quad
\ker D_{12}=\ov{{\bf i}_*\mat{c}{\ker T^* P^{\half}|_{\cX_0}\\ \{0\}}}.
\label{8.11}
\end{equation}
\end{lem}

\begin{proof} Let $\widetilde{\delta}\in\widetilde\Delta$ be such that
$D_{21}^*\widetilde{\delta}=0_{\cU}$. By construction \eqref{AIPcol1},
the vector
\begin{equation}
{\bf U}^*\widetilde\delta=\begin{bmatrix}C_2^*\widetilde\delta \\
D_{21}^*\widetilde{\delta} \\ 0\end{bmatrix}=\begin{bmatrix}x_0 \\ 0
\\ 0\end{bmatrix}\in\begin{bmatrix}\cX_0 \\ \cU \\ \widetilde{\Delta}_*
\end{bmatrix}
\label{8.12}
\end{equation}
belongs to $\Delta$, which means (by definition \eqref{3.4} of $\Delta$)
that
$$
0=\left\langle\begin{bmatrix}x_0 \\ 0 \\ 0\end{bmatrix},
\begin{bmatrix}P^\half x \\  Nx \\ 0\end{bmatrix}
\right\rangle_{\cX_0\oplus\cU\oplus\widetilde{\Delta}_*}=
\langle x_0, P^\half x\rangle_{\cX_0}
$$
for every $x\in\cX_0$, which is equivalent to $P^\half x_0=0$. Since
$P^\half|_{\cX_0}$ is injective, we get $x_0=0$. Thus
${\bf U}^*\widetilde\delta=0$ by \eqref{8.12} and consequently,
$\widetilde\delta=0$, since ${\bf U}$ is unitary. So ${\rm Ker}D_{21}^*=\{0\}$.

To prove the second equality in \eqref{8.11} we first observe that
a vector $\sbm{x_0 \\ y}\in\sbm{\cX_0\\ \cY}$ belongs to
$\Delta_*:=\sbm{\cX_0\\ \cY}\ominus{\mathcal R}_V$ if and only if
\begin{equation}
T^*P^\half x_0+E^*y=0.
\label{8.14}
\end{equation}
Now let $\widetilde{\delta}_*\in\widetilde\Delta_*$ so that $D_{12}\widetilde{\delta}_*=0_{\cY}$.
Then, by \eqref{AIPcol1} and \eqref{3.6},
\[
{\mathbf U}\wtil{\de}_*={\bf i}_*^{*}\wtil{\de}_*=\mat{c}{B_2\\D_{12}}\wtil{\de}_*
=\mat{c}{x_0\\0}\in\De_*\subset\mat{c}{\cX_0\\\cY},\quad\text{with}\ \ x_0=B_2\wtil{\de}_*.
\]
Since $\sbm{x_0\\0}$ is in $\De_*$, it follows from \eqref{8.14} that $T^*P^\half x_0=0$, i.e.,
$x_0\in\ker T^*P^\half$. Thus, since ${\bf U}$ is unitary and $\sbm{x_0\\0}\in\De_*$, we see that
$\wtil{\de}_*={\bf U}^*\sbm{x_0\\0}={\bf i}_*\sbm{x_0\\0}$ for some $x_0\in\ker T^*P^\half$.

Conversely, for every $x_0\in{\rm Ker}(T^*P^\half|_{\cX_0})$, the vector
$\sbm {x_0 \\ 0}$ belongs to $\Delta_*$ (by
\eqref{8.14}) so that its image $\widetilde\delta_*={\bf
i}_*\sbm{x_0 \\ 0}$ belongs to $\widetilde\Delta_*$
and we have, on account of \eqref{AIPcol1}-\eqref{3.6},
$$
\begin{bmatrix}B_2\\ D_{12}\end{bmatrix}\widetilde\delta_*=
\begin{bmatrix}B_2\\ D_{12}\end{bmatrix} {\bf i}_*\begin{bmatrix}x_0 \\
0\end{bmatrix}={\bf i}_*^*{\bf i}_*\begin{bmatrix}x_0 \\
0\end{bmatrix}=\begin{bmatrix}x_0 \\ 0\end{bmatrix}.
$$
Equating the bottom entries we get $D_{12}\widetilde\delta_*=0$ which
completes the proof of the characterization of ${\rm Ker}D_{12}$.
\end{proof}

\begin{thm}\label{T:InjAndIso}
Let $\Red_\Si$ be the Redheffer transform associated with the Schur class
function $\Si$ defined in \eqref{3.7} from the ${\AP}$-admissible data set
\eqref{3.1}. Assume that $T^*$ is injective. Then
$\Red_\Si:\cS(\widetilde{\Delta},\widetilde{\Delta}_*)\to\cS(\cU,\cY)$ is injective,
and for $\cE\in\cS(\widetilde{\Delta},\widetilde{\Delta}_*)$ the multiplication
operator $M_G:\cH(K_\cE)\to\cH(K_S)$ has trivial kernel. If in addition
$P=M_{F^{S}}^{[*]}M_{F^S}$, then $M_G$ is an isometry.
\end{thm}

\begin{proof}
It is easy to see that $M_{\Si_{21}}$ has dense range if and only if $\Si_{21}(0)=D_{21}$
has dense range and that $\ker \Si_{12}(0)=\ker D_{12}=\{0\}$ implies that $M_{\Si_{12}}$
has a trivial kernel. The converse of the latter statement is not true in general.
Moreover, the last statement of the theorem is immediate from
Proposition \ref{P:ContrMult}.
Thus Theorem \ref{T:InjAndIso} follows immediately from Lemma
\ref{L:kers}.
\end{proof}

\begin{rem}
{\rm
In case the operator $(I-\omega T)^{-1}$ is bounded for some
$\omega\in\D$, we can define
\begin{align*}
\widetilde{E}&=\sqrt{1-|\omega|^2}\cdot E(I-\omega T)^{-1},
&\widetilde{N}&=\sqrt{1-|\omega|^2}\cdot N(I-\omega T)^{-1}\\
\widetilde{T}&=(\overline{\omega}I-T)(I-\omega T)^{-1},
&\widetilde{S}(z)&=S\left(\frac{z-\omega}{1-z\overline{\omega}}\right).
\end{align*}
It is not hard to verify that if ${\mathcal D}=\{S, T, E, N, {\bf
x}\}$ is an $\AP$-admissible data set, then
the set $\widetilde{\mathcal D}=\{\widetilde{S}, \widetilde{T},
\widetilde{E}, \widetilde{N}, {\bf x}\}$
is also $\AP$-admissible and moreover, a function $f$ solves the problem
$\AP$ if and only $\widetilde{f}(z):=
f(\frac{z-\omega}{1-z\overline{\omega}})$ solves the problem ${\bf AIP}_{\cH(K_{\widetilde S})}$
with data set $\widetilde{\mathcal D}$. Therefore,
up to a suitable conformal change of variable, we get all the conclusions
in Theorem \ref{T:InjAndIso} under the assumption that $(I-\omega
T)^{-1}\in\cL(X)$ and $(T^*-\omega I)$ is injective for some
$\omega\in\D$.}
\label{R:4.4}
\end{rem}

In case $T^*$ is not injective, and neither is $(T^*-\omega I)$ for some $\omega\in\D$
so that $(I-\omega T)$ is invertible in $\cL(\cX)$, it may still be possible to reach
the conclusion of Theorem \ref{T:InjAndIso} under weaker assumptions on the operator $T$.
We start with a preliminary result.

\begin{lem}\label{L:2.3}
The operator $M_{\Sigma_{12}}: \, {\rm
Hol}_{\widetilde{\Delta}_*}(\D)\to {\rm
Hol}_{\cY}(\D)$ is not injective if and only if there is a
sequence $\{g_n\}_{n \ge 1}$ of non-zero vectors in $\cX_0$ such that
\begin{equation}
g_1\in{\rm Ker}(T^*P^\half), \quad P^\half g_n=T^*P^\half g_{n+1} \; \; (n\ge 1),
\quad \limsup_{n\to\infty}\|B_2^*g_{n+1}\|^{\frac{1}{n}}\le 1.
\label{8.20}
\end{equation}
\end{lem}

\begin{proof}
Let $h(z)={\displaystyle\sum_{k=0}^\infty z^k h_k}\in\tu{Hol}_{\wtil{\De}_*}(\D)$,
i.e.,
${\displaystyle\limsup_{k\to\infty}\|h_k\|^{\frac{1}{k}}}\leq 1$ and for $n\geq 0$,
\begin{equation}\label{hn}
h_n=i_*\mat{c}{x_n\\y_n}\text{ with }\mat{c}{x_n\\y_n}\in\mat{c}{\cX_0\\\cY},\text{ subject to }
T^*P^\half x_n+E^*y_n=0.
\end{equation}
Note that $x_n$ and $y_n$ are retrieved from $h_n$ by the identity
\begin{equation}\label{reviden}
\mat{c}{x_n\\y_n}={\mathbf i}_*^*h_n=\mat{c}{B_2\\D_{12}}h_n.
\end{equation}
Define $g_n\in\cX_0$ by
\begin{equation}\label{gns}
g_n=\sum_{k=0}^{n-1}A^{n-k-1}x_k\in\cX_0\quad\text{for}\quad n\geq 1,
\end{equation}
or equivalently via the recursion
\begin{equation}\label{gnsrec}
g_1=x_0,\quad g_{n+1}=x_n+Ag_n\quad\text{for}\quad n\geq 1.
\end{equation}
Since ${\mathbf U}$ in \eqref{AIPcol1} and \eqref{3.6} is unitary, it
follows that
\begin{equation}  \label{unirel1}
    B_{2}^{*}A = -D_{12}^{*} C_{1}.
\end{equation}
Moreover, since ${\mathbf U}$ is connected with $V$ as in
\eqref{AIPcol1} and $V$ is given by \eqref{3.3}, we see that
\begin{align*}
    \left\langle \left[ \begin{smallmatrix}  A \\ C_{1} \\ C_{2}
\end{smallmatrix} \right] x_{0}, \, \left[ \begin{smallmatrix}
P^{\half}T
\\ E \\ 0 \end{smallmatrix} \right] x_{0} \right\rangle  & =
\left\langle {\mathbf U} \left[ \begin{smallmatrix} x_{0} \\ 0 \\ 0
\end{smallmatrix} \right], \, {\mathbf U} \left[ \begin{smallmatrix}
P^{\half}x_{0} \\ N x_{0} \\ 0 \end{smallmatrix} \right] \right\rangle  \\
&  \quad = \left\langle \left[ \begin{smallmatrix} x_{0} \\ 0 \\ 0
\end{smallmatrix} \right], \, \left[ \begin{smallmatrix} P^{\half} x_{0}
\\ N x_{0} \\ 0 \end{smallmatrix} \right] \right\rangle = \langle
P^{\half} x_{0}, x_{0} \rangle
\end{align*}
from which we conclude that
\begin{equation}   \label{unirel2}
    T^{*} P^{\half}A + E^{*}C_{1} = P^{\half}.
\end{equation}
Hence for $n\geq 1$,
\begin{align}
T^*P^\half g_{n+1}=\sum_{k=0}^n T^*P^\half A^{n-k}x_k
&=T^*P^\half x_n+\sum_{k=0}^{n-1} T^*P^\half A^{n-k}x_k\notag\\
&=T^*P^\half x_n+\sum_{k=0}^{n-1}(P^\half-E^*C_1) A^{n-k-1}x_k\notag\\
&=T^*P^\half x_n-E^*C_1g_n+P^\half g_n\notag\\
&=-E^*(y_n+C_1g_n)+ P^\half g_n \label{com1}
\end{align}
where we used the relation between $x_n$ and $y_n$ in \eqref{hn} for
the last step.
Moreover, using the identity in \eqref{unirel1}, we get
\begin{align}
B_2^*g_{n+1}=B_2^{*}(x_n+Ag_n)
&=\mat{cc}{B_2^*&D_{12}^*}\mat{c}{x_n\\-C_1g_n}\notag\\
&=\mat{cc}{B_2^*&D_{12}^*}\mat{c}{x_n\\y_n}-D_{12}^*(y_n+C_1g_n)\notag\\
&={\mathbf i}_*\mat{c}{x_n\\y_n}-D_{12}^*(y_n+C_1g_n)\notag\\
&=h_{n} - D_{12}^*(y_n+C_1g_n).\label{com2}
\end{align}

Now assume that $h\not=0$ and $\Si_{12}(z)h(z)\equiv 0$, i.e.,
\begin{equation}\label{Si12h=0}
(D_{12}+zC_1(I-zA)^{-1}B_2)h(z)\equiv0.
\end{equation}
Using the power series representations
$$
h(z) = \sum_{n=0}^{\infty} h_{n} z^{n}\quad\mbox{and}\quad
\Si_{12}(z)=D_{12}+{\displaystyle\sum_{k=1}^\infty
z^kC_1A^{k-1}B_2}
$$
for $H$ and $\Sigma_{12}$ and recalling \eqref{reviden},
it follows that
\eqref{Si12h=0} is equivalent to the following system
of equations:
\begin{equation}\label{taylorform}
y_0=D_{12}h_0=0,\quad
y_n+C_1g_n=D_{12}h_n+\sum_{k=0}^{n-1}C_1A^{n-k-1}B_2h_k=0 \text{ for
} n \ge 1.
\end{equation}
Without loss of generality we may, and will, assume that $h_0\not=0$; otherwise replace $h$
by $\wtil{h}(z)=z^{-\ell}h(z)$ for $\ell\in\Z_+$ sufficiently large.
Then $\left[ \begin{smallmatrix} x_{0} \\ y_{0} \end{smallmatrix}
\right] = {\mathbf i}_{*}^{*} h_{0} \ne 0$ since $h_{0} \ne 0$ by assumption.
But $y_{0} = 0$ by \eqref{taylorform} and hence $x_{0} \ne 0$.  From
the constraint in \eqref{hn} we see that $0 \ne x_{0} \in
\operatorname{Ker} T^{*}P^{\half}$.  Moreover, the second identity in
\eqref{taylorform} combined with \eqref{com1} and \eqref{com2} yields
\begin{equation}  \label{id'}
    T^{*}P^{\half} g_{n+1} = P^{\half} g_{n}, \quad B_{2}^{*} g_{n+1} =
    h_{n}.
\end{equation}
The second of identities \eqref{id'} then gives us
\begin{equation}   \label{limsups}
\limsup_{n \to \infty} \| B_{2}^{*} g_{n+1} \|^{1/n} = \limsup_{n \to
\infty} \|h_{n}\|^{1/n} \le 1.
\end{equation}
Finally, observe that, since $g_{1} = x_{0} \ne 0$ and
$\operatorname{Ker} P^{\half}|_{\cX_{0}} = \{0\}$, the recursive
relation $T^{*}P^{\half} g_{n+1} = P^{\half}g_{n}$ (the first of
identities \eqref{id'}) implies that $g_{n} \ne 0$ for all $n \ge
1$.  We conclude that the sequence $\{g_{n}\}_{n \ge 1}$ has all the
desired properties.

Conversely, assume $\{g_n\}_{n\ge 1}$ is a sequence in $\cX_0$ satisfying \eqref{8.20}.
Define
\[
x_0=g_1,\quad x_n=g_{n+1}-Ag_n,\qquad y_0=0,\quad
y_n=-C_1g_n\quad\text{for }n\geq 1.
\]
Applying \eqref{com1} to $g_n$ for $n\geq 1$ and using \eqref{unirel2}, we find that
\begin{align*}
T^*P^\half g_{n+1}=P^\half g_n=&T^*P^\half Ag_n+E^*C_1 g_n\\
=&T^*P^\half  g_{n+1}-T^*P^\half x_n+E^*C_1g_n.
\end{align*}
We conclude that $T^*P^\half x_n+E^*y_n=T^*P^\half x_n-E^*C_1g_n=0$. For $n=0$ the
identity $T^*P^\half x_n+E^*y_n=0$ follows from the first of
conditions \eqref{8.20}. Hence we obtain that $\sbm{x_n\\y_n}\in\De_*$.

Now define $h_n=i_*\sbm{x_n\\y_n}\in\wtil{\De}_*$, and
$h(z)={\displaystyle\sum_{k=0}^\infty z^k h_k}$.
As before, $x_n$ and $y_n$ are  retrieved from $h_n$ by \eqref{reviden}, and from the
definition of $x_n$ it follows that the sequence $\{g_n\}_{n\ge 1}$ is
retrieved by \eqref{gnsrec}. Moreover, the definition of $y_n$ shows that
\eqref{taylorform}
holds and, in combination with the computation \eqref{com2}, that $B_2^* g_{n+1}=h_n$.
The latter implies that
${\displaystyle\limsup_{k\to\infty}\|h_k\|^{\frac{1}{k}}}\leq1$, via \eqref{8.20}
and the identities in \eqref{limsups}. In particular, $h\in\tu{Hol}_{\wtil{\De}_*}(\D)$.
The fact that \eqref{taylorform} holds now is equivalent to $\Si_{12}(z)h(z)\equiv0$. Note
that $x_0=g_1\not=0$ and hence $h_{0} = {\mathbf i}_{*} \left[ \begin{smallmatrix}
x_{0} \\ y_{0} \end{smallmatrix} \right] \ne 0$. Thus $h\not=0$ and it follows that
$M_{\Sigma_{12}}: \, {\rm Hol}_{\widetilde{\Delta}_*}(\D)\to {\rm Hol}_{\cY}(\D)$ is not
injective.
\end{proof}

Based on the previous result, we obtain the following relaxation of the the condition
on $T$ in Theorem \ref{T:InjAndIso}.

\begin{thm}  Let $\cD' = \{P,T,E,N\}$ be an $\APS$-admissible data set, let
$\Sigma$ be constructed as in \eqref{3.7} and let us assume that
$T$ meets the condition
\begin{equation}
\left(\bigcap_{k\ge 1}{\rm Ran} (T^*)^k\right)\bigcap{\rm Ker}T^*=\{0\}.
\label{8.30}
\end{equation}
Then the operator $M_{\Sigma_{12}}: \,
\operatorname{Hol}_{\widetilde{\Delta}_*}(\D)\to {\rm Hol}_{\cY}(\D)$ is injective.
\label{L:8.4}
\end{thm}
\begin{proof}
Assume $M_{\Si_{12}}: \, {\rm Hol}_{\widetilde{\Delta}_*}(\D)\to {\rm
Hol}_{\cY}(\D)$ is not injective. By Lemma \ref{L:2.3},
there exists a nonzero sequence $\{g_n\}_{n\geq 1}$ in $\cX_0$
satisfying \eqref{8.20}. By the first relation in \eqref{8.20},
$P^{\half}g_1\in\ker T^*$. On the other hand, iterating the second condition in \eqref{8.20}
gives $P^{\half}g_1=(T^*)^nP^\half g_{n+1}$ for each $n\geq 1$. Since $P^{\half}g_1\neq 0$,
it follows that $\tu{Ran} (T^*)^n  \cap \ker T^*\not=\{0\}$ for each $n\geq 1$.
The latter is in contradiction with \eqref{8.30}. Thus $M_{\Si_{12}}: \, {\rm
Hol}_{\widetilde{\Delta}_*}(\D)\to {\rm Hol}_{\cY}(\D)$ is injective.
\end{proof}
It is easy to see that not only the injectivity of $M_{\Si_{12}}$, but all the
conclusions of Theorem \ref{T:InjAndIso} hold with the condition that $T^*$ is injective
replaced by the weaker condition \eqref{8.30}. Although condition \eqref{8.30} is far from being
necessary, it guarantees injectivity of $M_{\Sigma_{12}}$ for important particular cases:
\begin{enumerate}
    \item $T^{*}$ is injective (so $\operatorname{Ker} T^{*} =
    \{0\}$),
    \item $T^{*}$ is nilpotent (so $\cap_{k\ge 1} \operatorname{Ran} (T^{*})^{k} =
    \{0\}$), and
    \item $\dim \cX < \infty$, or, more generally e.g., $T = \lambda
    I + K$ with $0 \ne \lambda \in {\mathbb C}$ and $K$ compact (so
    $\cX = \operatorname{Ran} (T^{*})^{p} \dot + \operatorname{Ker}
    (T^{*})^{p}$ once $p$ is sufficiently large).
\end{enumerate}
The question of finding a condition that is both necessary and sufficient for
injectivity of $M_{\Sigma_{12}}$ remains open.

\section{Description of all solutions of the problem $\AP$}\label{S:Param}

We now present the parametrization of the solution set to the problem $\AP$.
The proof relies on Theorem \ref{T:2.1}, Theorem \ref{T:3.1} and Proposition \ref{P:ContrMult}.
By Theorem \ref{T:2.1}, the solution set to the problem $\AP$ coincides with the set of all
functions $f:\D\to\cY$ such that the kernel ${\bf K}(z,\zeta)$ defined in \eqref{2.8} is positive on
$\D\times \D$. In particular, the function $S$ must be such that the kernel \eqref{3.2} is positive
meaning that $S$ must be a solution to the associated problem $\APS$. By Theorem \ref{T:3.1},
there exists a Schur-class function $\cE$ such that $S={\boldsymbol{\mathcal R}}_\Sigma[\cE]$
where ${\boldsymbol{\mathcal R}}_\Sigma$ is the Redheffer transform constructed in
\eqref{3.3}--\eqref{3.9}. Define $G$ and $\Gamma$ as in \eqref{GGa} and let $\wtil{\bf x}$ be the
unique vector in $\cX_0$ so that ${\bf x}=P^\half\wtil{\bf x}$. Making use of equalities
\eqref{3.14'} and \eqref{3.14b} we can write ${\bf K}(z,\zeta)$ in the form
$$
{\bf K}(z,\zeta)=
\begin{bmatrix} 1 & \widetilde{\bf x}^*P^\half & f(\zeta)^*\\
P^\half\widetilde{\bf x} & P & P^\half\Gamma(\zeta)^{*} \\ f(z) &
\Gamma(z)P^\half & G(z)K_\cE(z,\zeta)G(\zeta)^*+\Gamma(z)\Gamma(\zeta)^*
\end{bmatrix}.
$$
The positivity of the latter kernel is equivalent to positivity of the Schur complement of $P$
with respect to ${\bf K}(z,\zeta)$, that is, to the condition
\begin{equation}
\begin{bmatrix}1-\|\widetilde{\bf x}\|^2 & f(\zeta)^*-\widetilde{\bf
x}^*\Gamma(\zeta)^{*} \\
f(z)-\Gamma(z)\widetilde{\bf x} &  G(z)K_\cE(z,\zeta)G(\zeta)^*
\end{bmatrix}\succeq 0\qquad (z,\zeta\in\D).
\label{3.17}
\end{equation}
We arrive at the following result.
\begin{thm}\label{T:AIPsol}
Let $\{S, T, E, N, {\bf x}\}$ be an $\AP$-admissible data set and let us assume that
$P:=M_{F^S}^{[*]}M_{F^S}\ge {\bf x}{\bf x}^*$ with $F^S$ as in \eqref{1.17}.
Let $\Sigma$ be constructed as in \eqref{3.3}--\eqref{RedCoefs2}, let
$\cE$ be a Schur-class function such that $S={\boldsymbol{\mathcal
R}}_\Sigma[\cE]$, let $G$ and $\Ga$ be defined as in \eqref{GGa} and let
$\wtil{\bf x}$ be the unique vector in $\cX_0$ so that ${\bf x}=P^\half\wtil{\bf
x}$. Then:
\begin{enumerate}
\item The set of solutions $f$ of the problem $\AP$ is given by the formula
\begin{equation}\label{sols}
f(z)=\Ga(z)\wtil{\bf x}+G(z)h(z)
\end{equation}
with parameter $h$ in $\cH(K_\cE)$ subject to $\|h\|_{\cH(K_S)}\leq \sqrt{1-\|\wtil{\bf
x}\|^2}$.
\item For $f$ defined by \eqref{sols}
\begin{equation}\label{OrthoRelation}
\|f\|^2_{\cH(K_S)}=\|M_\Ga\wtil{\bf x}\|^2+\|M_Gh\|^2
=\|\wtil{\bf x}\|^2+\|P_{\cH(K_\cE)\ominus\ker M_G}h\|^2
\end{equation}
and hence $f_\textup{min}(z)=\Ga(z)\wtil{\bf x}$ is the unique minimal-norm solution.
\item The problem $\AP$ admits a unique solution if and only if $\|\wtil{\bx}\|=1$ or
$\ov{\tu{Ran}}\,M_{F^S}=\cH(K_S)$.
\end{enumerate}
\end{thm}

\begin{proof} As we have seen, a function $f: \D\to\cY$ solves the problem $\AP$
if and only if \eqref{3.17} holds, that is, if and only if the function
$g:=f-M_\Gamma\widetilde{\bf x}$
belongs to the reproducing kernel Hilbert space $\cH(\widetilde{K})$ with
reproducing kernel $\widetilde{K}(z,\zeta)=G(z)K_\cE(z,\zeta)G(\zeta)^*$
and satisfies $\|g\|_{\cH(\widetilde{K})}\le \sqrt{1-\|\widetilde{\bf
x}\|^2}$. The range characterization of $\cH(\widetilde{K})$ tells us that
$$
\cH(\widetilde{K})=\{G(z)h(z): \; h\in\cH(K_\cE)\} \; \mbox{with norm} \;
\|M_Gh\|_{\cH(\widetilde{K})}=\|(I-{\bf q})h\|_{\cH(K_\cE)}
$$
where ${\bf q}$ is the orthogonal projection onto the subspace
$\ker M_G\subset \cH(K_\cE)$. Therefore, the function
$g=f-M_\Gamma\widetilde{\bf x}$ is of the form $g=M_Gh$ for some
$h\in\cH(K_\cE)$ such that
$\|h\|_{\cH(K_\cE)}=\|g\|_{\cH(\widetilde{K})}\le \sqrt{1-\|\widetilde{\bf
x}\|^2}$. This proves the characterization of solutions through \eqref{sols}.

Since $P=M_{F^S}^{[*]}M_{F^S}$, it follows from Proposition
\ref{P:ContrMult} that the operator \eqref{3.14} is a coisometry and $M_\Ga$
is an isometry. From this combination the orthogonality between the minimal-norm
solution $f_\textup{min}(z)=\Ga(z)\wtil{\bf x}$ and the remainder on the right hand
side of \eqref{sols}, as well as the second identity in \eqref{OrthoRelation},
is evident.

Since $M_G$ is a partial isometry, it follows from \eqref{sols} that
the problem $\AP$ admits a unique solution if and only if either $\|\wtil{\bx}\|=1$
(because then $h=0\in\cH(K_\cE)$ is the only admissible parameter), $\cH(K_\cE)=\{0\}$ (i.e., if
$\cE$ is an unimodular constant) or $M_G=0$. Since the operator \eqref{3.14}
is a coisometry and because $M_\Ga$ is an isometry, the last two cases are
covered by the condition that $M_\Ga$ is unitary. Due to the relation
between $F^{S}$ and $\Ga$ (see \eqref{3.14'}), this is equivalent to $M_{F^S}$ having dense range.
\end{proof}

Although the correspondence $\cE\to S={\boldsymbol{\mathcal
R}}_\Sigma[\cE]$ established by formula \eqref{3.9} is not one-to-one in
general, it follows from the proof of Theorem \ref{T:AIPsol} that
in order to find all solutions $f$ of the problem $\AP$ it suffices
to take into account just one parameter $\cE$ so that
$S={\boldsymbol{\mathcal R}}_\Sigma[\cE]$, rather than all.
The further analysis in Section \ref{S:Redheffer}, i.e., Theorem \ref{T:InjAndIso} and
Lemma \ref{L:8.4}, provide conditions under which the Schur class function $\cE$ in Theorem
\ref{T:AIPsol} is unique.

\begin{thm}
Let \eqref{1.2a} be an $\AP$-admissible data set satisfying condition
\eqref{2.6} and assume that the operator $T^*$ satisfies condition \eqref{8.30}. Then:
\begin{enumerate}
\item There is a unique Schur-class function $\cE$ such that
$S={\boldsymbol{\mathcal R}}_\Sigma[\cE]$, where ${\boldsymbol{\mathcal
R}}_\Sigma$ is the Redheffer transform constructed from the data set
\eqref{1.2a} via \eqref{3.3}--\eqref{RedCoefs2}.

\item The parametrization $h\mapsto f$ in Theorem \ref{T:AIPsol}, via formula
\eqref{sols}, of the solutions $f$ to the problem $\AP$ is injective. That is,
the operator $M_G: \, \cH(K_\cE)\to  \cH(K_S)$ is isometric so that in addition
\begin{equation}
\|f\|^2_{_{\cH(K_S)}}=\|\widetilde{\bf
x}\|^2_{_{\cX_0}}+\|h\|^2_{_{\cH(K_S)}}.
\label{3.20a}
\end{equation}
\end{enumerate}
\label{T:3.11}
\end{thm}

\begin{rem} {\rm
Given an $\AP$-admissible data set $(S,E,N,T,\bx)$, it is straightforward
that $(E,N,T,P)$ with $P=F^{S[*]}F^{S}$ is an $\APS$-admissible data set
and that $S$ is a solution for the associated problem $\APS$. Now consider
another solution $\wtil{S}\in\cS(\cU,\cY)$ of the problem $\APS$. Unlike
for $S$, this solution $\wtil{S}$ satisfies $M_{F^{\wtil{S}}}^{[*]}M_{F^{\wtil{S}}}\leq P$
and equality may not hold. We may then still ask the question for
which functions $f:\BD\to\cY$ the kernel in \eqref{2.8} is positive, with $S$
is replaced by $\wtil{S}$. This question turns out to to be equivalent to
that of determining the $f\in\cH(K_{\wtil{S}})$ with $\|f\|_{\cH(K_{\wtil{S}})}\le 1$ and
such that the vector $M_{F^{\wtil{S}}}^{[*]}f$ is close to ${\bf x}$ in the sense that
\begin{equation}
M_{F^{\wtil{S}}}^{[*]}f={\bf
x}+\sqrt{1-\|f\|^2_{\cH(K_{\wtil{S}})}}\left(P-M_{F^{\wtil{S}}}^{[*]}
M_{F^{\wtil{S}}}\right)^{\frac{1}{2}}\widehat{\bf x}
\label{3.2a}
\end{equation}
for some $\widehat{\bf x}\in\cX$ with $ \|\widehat{\bf x}\|\le 1$. The solutions
to this problem can still be parameterized by formula \eqref{sols}, with now
$\cE\in\cS(\wtil{\De},\wtil{\De}_*)$ so that $\wtil{S}={\boldsymbol{\mathcal R}}_\Sigma[\cE]$
in the definition of $\Ga$ and $G$, with the twist that in this case, because we may not have
$M_{F^{\wtil{S}}}^{[*]}M_{F^{\wtil{S}}}= P$, there is no guarantee that we have orthogonality
as in \eqref{OrthoRelation}, nor is it clear if the `central' solution $f=M_{\Ga}\wtil{\bf x}$
is the solution with minimal norm.}
\end{rem}

To conclude this section we will briefly discuss the interplay between
the uniqueness of $S$ as a solution of the problem $\APS$ (with
$P = M_{F^{S}}^{[*]} M_{F^{S}}$ of the form \eqref{2.6}) and the determinacy of the related
(unconstrained) problem $\AP$. We will assume that the operator $T$
meets the condition \eqref{8.30}, leaving the general case open.
Under this assumption, there are only three uniqueness and
semi-uniqueness cases. Recall that $\Delta$ and $\Delta_*$ are the defect
spaces of the isometry \eqref{3.3}.\medskip

{\bf Case 1:} Let $\Delta_*=\{0\}$. Then
$S=\Sigma_{11}$ is the unique solution of the problem $\APS$.
Furthermore, we conclude from \eqref{GGa} that
$$
\Gamma(z)=C_1(I-zA)^{-1},\quad G(z)\equiv 0\quad\mbox{and}\quad
\cH(K_\cE)=\{0\}.
$$
 By Theorem \ref{T:AIPsol}, the unconstrained problem
$\AP$ has a unique solution $f(z)=C_1(I-zA)^{-1}\widetilde{\bf x}$ where
$\widetilde{\bf x}$ is the unique vector in $\cX_0$ such that
$P^\half\widetilde{\bf x}={\bf x}$.\medskip

{\bf Case 2:} Let $\Delta=\{0\}$. In this case still $S=\Sigma_{11}$ is
the unique solution of the problem $\APS$. Also we have
$\Gamma(z)=C_1(I-zA)^{-1}$. However, we now have
$G=\Sigma_{12}$ and $\cH(K_\cE)=H^2_{\widetilde{\Delta}_*}$. By Theorem
\ref{T:AIPsol}, all
solutions $f$ to the unconstrained problem $\AP$ are given by
\begin{equation}
f(z)=C_1(I-zA)^{-1}\widetilde{\bf x}+\Sigma_{12}(z)h(z),
\label{8.15}
\end{equation}
where $\widetilde{\bf x}$ is as above  and where $h$ varies in
$H^2_{\widetilde{\Delta}_*}$. One can see that the same description holds
if the spaces \eqref{3.4} are nontrivial and $S=\Sigma_{11}$ is the
central (but not unique) solution to the associated problem $\APS$.\medskip

{\bf Case 3:} Let  $\Delta$ and $\Delta_*$ be nontrivial and let us
assume that $S$ is an extremal solution to the problem $\APS$
(in the sense that the unique $\cE$ such that $S={\boldsymbol{\mathcal
R}}_{\Sigma}[\cE]$ is a coisometric constant). Then the unconstrained
problem $\AP$ has a unique solution since in this case
$H(K_{\cE})=\{0\}$.

\section{Interpolation with operator argument}\label{S:OAP}

In this section we show that the interpolation problem with operator
argument in the space $\cH(K_S)$ can be embedded into the general scheme
of the problem $\AP$ considered above. Recall that a pair $(E,T)$ with
$E\in\cL(\cY,\cX)$ and $T\in\cL(\cX)$ is called an analytic output pair if
the observability operator $\cO_{E,T}$ maps $\cX$ into
$\operatorname{Hol}_{\cY}(\D)$. The starting point for the
operator-argument point-evaluation is a so-called {\em output-stable} pair $(E,T)$
which is an analytic output pair with the additional property that
$\cO_{E,T}\in\cL(\cX,H_\cY^2)$:
\begin{equation}
\cO_{E,T}: \; x\mapsto E(I-zT)^{-1}x=\sum_{n=0}^\infty z^n ET^nx\in
H^2_\cY.
\label{3.1a}
\end{equation}
Given such an output-stable pair $(E,T)$ and a
function $f\in H^2_\cY$, we  define the
    {\em left-tangential operator-argument point-evaluation}
$(E^{*}f)^{\wedge  L}(T^{*})$
    of $f$ at $(E,T)$ by
    \begin{equation}\label{1.4}
    (E^{*} f)^{\wedge L}(T^{*}) = \sum_{n=0}^\infty
    T^{* n} E^{*} f_{n}\quad \text{if} \quad f(z) =
    \sum_{n=0}^\infty f_{n} z^{n}.
    \end{equation}
    The computation
$$
    \left\langle \sum_{n=0}^\infty T^{* n}E^{*}
    f_{n}, \; x \right \rangle_{\cX}  =
    \sum_{n=0}^\infty \left\langle f_{n}, \; E T^{n}
            x \right \rangle_{\cY}
             = \langle f, \; \mathcal {O}_{E, T} x\rangle_{H^2_{\cY}}
$$
    shows that the output-stability of the pair $(E, T)$ is exactly
    what is needed for the infinite series in the definition
    of $(E^{*}f)^{\wedge L}(T^{*})$ in \eqref{1.4} to converge
    in the weak topology on $\cX$.
The same computation shows that tangential evaluation
with operator argument amounts to the adjoint of $\cO_{E, T}$:
\begin{equation}
(E^{*}f)^{\wedge  L}(T^{*})= \cO_{E, T}^{*}f\quad\mbox{for}\quad f \in
H_\cY^2.
\label{1.6}
\end{equation}
Evaluation \eqref{1.4} applies to functions from de
Branges-Rovnyak spaces $\cH(K_S)$ as well, since $\cH(K_S)\subset H_\cY^2$,
and suggests the following interpolation problem.

\medskip

$\OAP$: {\em Given $S\in\cS(\cU,\cY)$, $T\in\cL(\cX)$, $E\in\cL(\cY,\cX)$
and ${\bf x}\in\cX$ so that the pair $(E,T)$ is output stable, find all
functions
$f\in\cH(K_S)$ such that
\begin{equation}
\|f\|_{_{\cH(K_S)}}\le 1\quad\mbox{and}\quad(E^{*}f)^{\wedge
L}(T^{*})=\cO_{E, T}^{*}f={\bf x}.
\label{1.7}
\end{equation}}
In the scalar-valued case $\cU=\cY=\C$, the latter problem has been
considered recently in \cite{bbt2}, with the additional assumption that $P>0$.
Similarly to the situation in \cite{bbt2},
the operator-valued version contains left-tangential Nevanlinna-Pick and
Carath\'eodory-Fej\'er interpolation problems as particular cases
corresponding to special choices of $E$ and $T$. We now show that
on the other hand, the problem $\OAP$ can be considered as a particular
case of the problem $\AP$.

\begin{lem}
Let $(E,T)$ be an output stable pair with $E\in\cL(\cY,\cX)$ and
$T\in\cL(\cX)$, let $S\in\cS(\cU,\cY)$ be a Schur-class function and let
$N\in\cL(\cX,\cU)$ be defined by
\begin{equation}\label{Ndef}
N:=\sum_{j=0}^\infty S_j^*ET^j,\quad\mbox{where}\quad
    S(z)=\sum_{j=0}^\infty S_j z^j
\end{equation}
or equivalently, via its adjoint
\begin{equation}
N^*=\cO_{E, T}^{*}M_S\vert_\cU: \, \cU\to\cX.
\label{1.8}
\end{equation}
Then the data set
${\mathcal D}=\{S, \, T, \, E, \, N, \, {\bf x}\}$
is $\AP$-admissible for every ${\bf x}\in\cX$.
Furthermore, $M^{[*]}_{F^S}=\cO_{E,T}^*|_{\cH(K_S)}$, so that the interpolation
conditions \eqref{1.7} coincide with those in \eqref{1.3a}.
\label{L:5.2}
\end{lem}

\begin{proof} For $N$ defined as in \eqref{Ndef}, the pair
$(N,T)$ is output stable (cf. \cite[Proposition  3.1]{bbieot})
and the observability operator $\cO_{N,T}: \; x\mapsto
N(I-zT)^{-1}x$ equals
\begin{equation}
\cO_{N, T}=M_S^*\cO_{E, T}: \; \cX\to H_\cU^2.
\label{1.9}
\end{equation}
With $N$ as above, we now define $F^S$ by formula \eqref{1.17}.
For the multiplication operator \eqref{1.18a} we have, on account of
\eqref{1.9},
\begin{equation}
M_{F^S}=\cO_{E,T}-M_S\cO_{N,T}=(I-M_S^*M_S)\cO_{E,T}
\label{1.18}
\end{equation}
which together with the  range characterization of $\cH(K_S)$ implies
that $M_{F^S}$ maps $\cX$ into $\cH(K_S)$. Furthermore, it follows from
\eqref{1.17}, \eqref{1.2} and \eqref{1.4} that
\begin{align*}
\|F^Sx\|^2_{\cH(K_S)}=&
\langle (I - M_S M_S^{*}) \cO_{E,T}x, \cO_{E,T}x\rangle_{H_\cY^2}\\
=&\langle (\cO_{E,T}^{*} \cO_{E,T} -
          \cO_{N,T}^{*}\cO_{N,T})x, x \rangle_{\cX}
\end{align*}
for every $x\in\cX$. The latter equality can be written in operator form
as
\begin{equation}
P:=M_{F^S}^{[*]}M_{F^S}=\cO_{E,T}^{*} \cO_{E,T} -
          \cO_{N,T}^{*}\cO_{N,T}.
\label{1.10}
\end{equation}
It follows from the series representation \eqref{3.1a} and the definition
of inner product in $H^2_\cY$ that
$$
\cO_{E,T}^{*} \cO_{E,T}=\sum_{n=0}^\infty T^{*n}E^*ET^n
$$
with convergence in the strong operator topology. Using the latter  series
expansion one can easily verify the identity
\begin{equation}
\cO_{E,T}^{*} \cO_{E,T}-T^*\cO_{E,T}^{*} \cO_{E,T}T=E^*E.
\label{1.20}
\end{equation}
Since the pair $(N,T)$ is also output stable, we have similarly
\begin{equation}
\cO_{N,T}^{*} \cO_{N,T}-T^*\cO_{N,T}^{*} \cO_{N,T}T=N^*N.
\label{1.21}
\end{equation}
Subtracting \eqref{1.21} from \eqref{1.20} and taking into account
\eqref{1.10} we conclude that $P$ satisfies the Stein identity
\eqref{1.11}. Thus, the data set ${\mathcal D}$ is $\AP$-admissible.
In view of \eqref{1.18} and \eqref{1.2}, the equalities
 \begin{align*}
    \langle M_{F^S}^{[*]} f, \, x\rangle_{\cX}
    =\langle f, \,M_{F^S} x\rangle_{\cH(K_S)}
    =&\langle f, \, (I-M_SM_S^*)\cO_{E,T}x\rangle_{\cH(K_S)}
    \notag \\
    =&\langle f, \, \cO_{E,T}x\rangle_{H^2_\cY}=\langle
    \cO_{E,T}^*f, \, x\rangle_{\cX}\notag
    \end{align*}
hold for all $f\in\cH(K_S)$ and $x\in\cX$. Therefore,
$M_{F^S}^{[*]}=\cO_{E,T}^*|_{\cH(K_S)}$.
\end{proof}

As a consequence of Lemma \ref{L:5.2}, the solutions to the problem $\OAP$ are
obtained from Theorem \ref{T:AIPsol}, after specialization to the case under
consideration. We do not state this specialization of Theorem \ref{T:AIPsol} here
because the formulas do not significantly simplify. Instead we now discuss
the operator-argument interpolation problem for functions in $H^2_\cY$, that is, the
problem $\OAP$ with $S\equiv0$. As we shall see, in that case the problems $\AP$
and $\OAP$ coincide.

Consider an $\AP$-admissible data set $\{S, \, T, \, E, \, N, \, {\bf x}\}$ with
$S\equiv0\in\cS(\cU,\cY)$. Then $\cH(K_S)=H^2_\cY$ and $F^S=\cO_{E,T}$. Thus
condition (2) in Definition \ref{D:1.1} just says that $\FS=\cO_{E,T}$ is in
$\cL(\cX,H^2_\cY)$, and thus
that $(E,T)$ is output-stable. The third condition states that
$P=\FSstar\FS=\cO_{E,T}^*\cO_{E,T}$ satisfies the Stein equation \eqref{1.11}.
This implies that necessarily $N^*=0=(E^* S)^{\wedge L}(T^*)$, and it follows that
the problem ${\AP}$ reduces to the problem $\OAP$ with data $T$, $E$ and ${\bf x}$,
and $S\equiv0$. We now specify Theorem \ref{T:AIPsol} to this case, with the additional
assumption that $P$ is positive definite.

\begin{thm}\label{T:5.1}
Given an output stable pair $(E,T)$ with $E\in\cL(\cY,\cX)$ and $T\in\cL(\cX)$, and
${\bf x}\in\cX$. Assume that ${\bf x}{\bf x}^*\leq P:=\cO_{E,T}^*\cO_{E,T}$
and that $P$ is positive definite. Then the set of all
$f\in\cH^2_\cY$ satisfying
\[
\|f\|_{H^2_\cY}\leq 1\quad\text{and}\quad (E^* f)^{L}(T^*)={\bf x}
\]
is given by the formula
\begin{equation}\label{H2param}
f(z)=E(I-zT)^{-1}P^{-1}\bx+B(z)h(z)
\end{equation}
where $h$ is a free parameter from the ball
$$
\left\{ h \in H^{2}_{\cY_{0}} \colon \|h\|^2_{H^2_{\cY_{0}}}\le
1- \bx^{*}P^{-1}\bx \right\}\subset H^2_\cY
$$
for an auxiliary Hilbert space $\cY_0$; here
 $B(z)$ is the
inner function in the Schur class $\cS(\cY_{0}, \cY)$  determined uniquely
(up to a constant unitary factor on the
right) by the identity
\begin{equation}   \label{innerB}
K_B(z,\zeta):=\frac{I_\cY -B(z)B(\zeta)^{*}}{1
-z\overline{\zeta}}=E(I- zT)^{-1} P^{-1}(I -
\overline{\zeta}T^{*})^{-1}E^*.
\end{equation}
\end{thm}

\begin{proof}
As remarked above, we are considering the problem $\AP$ with
data set $\{S, \, T, \, E, \, N, \, {\bf x}\}$ where $S\equiv0$ and $N=0$.
Then, for $x \in \cX$, we have
\begin{align*}
    \langle P x, x \rangle
    =& \lim_{N \to \infty} \left\langle \sum_{n=0}^{N-1} T^{*n} E^{*}E
    T^{n} x, x \right\rangle \\
   = &  \lim_{N \to \infty} \left\langle \sum_{n=0}^{N-1} T^{*n} (P -
    T^{*}P T)T^{n} x, x \right\rangle\\
     =& \langle P x, x \rangle - \lim_{N \to \infty} \| P^{\half} T^{N}
    x \|^{2}.
\end{align*}
and we conclude that $\| P^{\half} T^{N} x\|^{2} \to 0$ as $N \to
\infty$.  The assumption that $P > 0$ implies that $P^{\half}$ is
invertible and we conclude that $\|T^{N}x\|^{2} \to 0$ as well, i.e.,
that $T$ is strongly stable.

The fact that $N=0$ yields that in the construction of the unitary colligation ${\bf U}$
in \eqref{3.3}--\eqref{3.6}, $\cD_V=\cX$, $\De=\cU$ and the isometry
$V$ is defined by the identity $VP^\half=\sbm{P^\half T\\E}$. Moreover, in the
unitary colligation ${\bf U}$ we have $B_1=0$, $D_{11}=0$
and $C_2=0$, and $A$ and $C_{1}$ can be computed explicitly as
$$
A=P^\half TP^{-\half}, \quad C_{1} = E P^{-\half}.
$$
As $T$ is strongly stable, we conclude that $A$ is strongly stable as
well. The unitary colligation $\bU$ then collapses to
\begin{equation}   \label{Uspecial}
 \bU = \begin{bmatrix} A & 0 & B_{2} \\ C_{1} & 0 & D_{12} \\ 0 &
 D_{21} & 0 \end{bmatrix}
\end{equation}
and $\Sigma(z)$ has the form
$$
\Sigma(z) = \begin{bmatrix} \Sigma_{11}(z) & \Sigma_{12}(z) \\
\Sigma_{21}(z) & \Sigma_{22}(z) \end{bmatrix} = \begin{bmatrix} 0 &
D_{12} + zC_{1}(I - zA)^{-1} B_{2} \\
 D_{21} & 0 \end{bmatrix}.
$$
{}From the special form \eqref{Uspecial} of $\bU$, it follows that
$D_{21}$ and $ \begin{bmatrix} A & B_{2} \\ C_{1} & D_{12} \end{bmatrix}$
are unitary. As $A$ is strongly stable, it is then well known that
$\Sigma_{12}$ is inner, $\cO_{C_{1},A}$ maps $\cX$ isometrically into
the de Branges-Rovnyak space $\cH(K_{\Sigma_{12}}) = H^{2}_{\cY}
\ominus \Sigma_{12} H^{2}_{\widetilde \De_{*}}$, and hence the
operator
\begin{equation} \label{unitaryguy}
\begin{bmatrix} M_{\Sigma_{12}} & \cO_{C_{1},A} \end{bmatrix} \colon
    \begin{bmatrix} H^{2}_{\widetilde \Delta_{*}} \\ \cX
    \end{bmatrix} \to H^{2}_{\cY}
\end{equation}
is unitary.
 Note that
the Redheffer transform $\cR_\Si$ reduces to
\[
S(z)=\cR_\Si[\cE](z)=\Si_{12}(z)\cE(z)D_{21}\quad (z\in\BD).
\]
Since $S\equiv 0$, we have $S=\cR_\Si[\cE]$ when $\cE\equiv0$. In fact, because $\Si_{12}$
is inner and $D_{21}$ unitary, the Redheffer transform $\cR_\Si$ is one-to-one, and
thus $\cE\equiv0$ is the only $\cE\in\cS(\De,\De_*)$ with $\cR_\Si[\cE]\equiv0$.
Then $\cH(K_\cE)=H^2_{\De_*}$ and the function $G$ in \eqref{GGa} is equal
to $\Si_{12}$, and thus is inner. To complete the proof, note that $\Gamma=\FS P^{-\half}$
and $\wtil{\bx}=P^{-\half}\bx$, so that
\begin{align*}
\Gamma(z)\wtil{\bx}=&\FS(z) P^{-\half}P^{-\half}\bx\\
=&\cO_{E,T}P^{-1}\bx=E(I-zT)^{-1}P^{-1}x
= C_{1} (I - zA)^{-1} x.
\end{align*}
Thus \eqref{sols} coincides with \eqref{H2param} with $B=\Si_{12}$.
The coisometric property of the  unitary operator \eqref{unitaryguy}
expressed in reproducing kernel form gives us
\begin{align*}
\frac{ I - \Sigma_{12}(z) \Sigma_{12}(\zeta)^{*}}{1 - z
\overline{\zeta}} =& C_{1}(I - zA)^{-1} (I - \overline{\zeta}A^{*})^{-1}C_1^*\\
   =& E (I - zT)^{-1} P^{-1} (I - \overline{\zeta} T^{*})^{-1} E^{*}
\end{align*}
and we see that $B: = \Sigma_{12}$ is determined from the data set as in
\eqref{innerB} in Theorem \ref{T:5.1}.
\end{proof}

\section{Homogeneous interpolation and Toeplitz kernels}
\label{S:homoint}
    \setcounter{equation}{0}

    Let  $S\in\cS(\cU,\cY)$ be an inner function, i.e.,
$M_S\in\cL(H^2_\cU,H^2_\cY)$ is an isometry.
    Then $M_S H^2_\cU$ is a closed, invariant subspace of the shift
operator $M_z$ on $H^2_\cY$.
    By the Beurling-Lax-Halmos theorem, this is the general form of a
closed shift-invariant
    subspace of $H^2_\cY$. Moreover, the de Branges-Rovnyak space
$\cK_S:=\cH(K_S)$ is the orthogonal
    complement of $M_S H^2_{\cU}$:
    $$
    \cK_S=H^2_\cY\ominus M_S H^2_{\cU}
    $$
    and provides a general form for closed backward shift-invariant
subspaces of $H^2_\cY$. Let, in addition, $B\in\cS(\cW,\cY)$ be inner, so that
we have shift invariant subspaces $M_S H^2_{\cU}$ and $M_B H^2_{\cW}$ and
backward shift invariant subspaces $\cK_S$ and $\cK_B$ of $H^2_\cY$.
Characterizations of the intersections $M_S H^2_{\cU}\cap M_B H^2_{\cW}$
and $\cK_S\cap\cK_B$ in
    terms of $S$ and $B$ are well-known
    (see e.g., \cite{N}). In this section we characterize the space
    \begin{equation}\label{intersect}
    \cM_{S,B}:=\cK_S\cap M_B H^2_{\cW}.
    \end{equation}

Let us introduce the operators $T\in\cL({\mathcal K}_B)$,
$E\in\cL({\mathcal K}_B,\cY)$, and $N\in \cL(\cK_{B},\cU)$ by
    \begin{align}
       &  T \colon h(z) \mapsto \frac{h(z)-h(0)}{z},\quad
        E \colon h \to h(0),\label{defT} \\
      &  N \colon h(z)=\sum_{j\ge 0}h_jz^j\mapsto \sum_{j\ge 0}S_j^*h_j
      \quad\text{where}\quad S(z) = \sum_{j \ge 0} S_{j} z^{j}.
       \label{defn1}
        \end{align}
The operator $T$ is strongly stable (i.e.,
${\displaystyle\lim_{n\to\infty} T^n h=0}$ for each $h\in\cX=\cK_B)$ and
the pair $(E,T)$ is output-stable. With $N$ defined in accordance
with \eqref{Ndef}, the data set ${\mathcal D}=\{S,E,N,T, {\bf x}=0\}$ is
$\AP$-admissible, by Lemma \ref{L:5.2}. Furthermore the adjoint
$\cO_{E,T}^*: \, H^2_{\cY}\to{\mathcal K}_B$  of the observability
operator $\cO_{E,T}$ amounts to the orthogonal projection
$P_{{\mathcal K}_B}$ onto ${\mathcal K}_B$. Indeed, if $h(z)=\sum_{j\ge
0}h_jz^j\in{\mathcal K}_B$, then $ET^jh=h_j$ for $j\ge 0$ and hence
    $$
    (\cO_{E,T}h)(z)=\sum_{j\ge 0}(ET^jh)z^j=\sum_{j\ge 0}h_jz^j=h(z).
    $$
    Therefore, for an $f\in H^2_{\cY}$ we have
    $\; \cO_{E,T}^*f=0 $ if and only if $ f\in
    H^2_\cY\ominus {\mathcal K}_B=M_B H^2_{\cW}$.
    It is now easily checked that
     the space \eqref{intersect} is characterized as
    \begin{equation}
    {\mathcal M}_{S,B}=\left\{f\in{\mathcal K}_S: \;
    \cO_{E,T}^*f=0\right\},\label{hom}
    \end{equation}
    i.e., as the solution set of the (unconstrained) homogeneous problem
    {\bf OAP}$_{\cH(K_S)}$ with the data set $\{S,E,T, {\bf x}=0$\}.
    The operator $P$ defined by formulas \eqref{1.10} now
amounts to the compression of the operator $I_{H^2_\cY}-M_SM_S^*$ to the
subspace ${\mathcal K}_B$:
\begin{equation}
    P=I_{{\mathcal K}_B}-P_{{\mathcal K}_B}M_SM_S^*\vert_{{\mathcal K}_B}.
\label{defP}
\end{equation}
\begin{thm}  \label{T:homoint}
Given inner functions $S \in \cS(\cU, \cY)$ and $B \in \cS(\cW, \cY)$,
let $\Sigma=\sbm{\Sigma_{11} & \Sigma_{12}\\ \Sigma_{21}& \Sigma_{22}}$
be the characteristic function of the unitary colligation ${\bf U}$
associated via formulas \eqref{3.3}--\eqref{3.6} to the tuple
$\{P,T,E,N\}$ given in \eqref{defT}, \eqref{defn1}, \eqref{defP}. Then
the space $\cM_{S,B}$ given by \eqref{intersect}
is given explicitly as
       \begin{equation} \label{IntersectForm}
         \cM_{S,B} = G \cdot \cH(K_{\cE})
        \end{equation}
where $\cE$ is the unique function in $\cS(\cU \oplus \widetilde
\Delta_{*})$ such that $S={\boldsymbol{\mathcal R}}_{\Sigma}[\cE]$
and $G(z) =  \Sigma_{12}(z) (I - \cE(z) \Sigma_{22}(z))^{-1}$.
Furthermore $M_G$ is a unitary operator from $\cH(K_{\cE})$ onto
$\cM_{S,B}$.
\end{thm}
\begin{proof}
The parametrization formula \eqref{IntersectForm} follows from
\eqref{hom} upon applying Theorem \ref{T:AIPsol}. The fact that
in the present situation the Redheffer transform ${\boldsymbol{\mathcal R}}_{\Sigma}$
is one-to-one was
established in \cite{khdep} (see also \cite[Theorem 5.8]{khy}).
Thus the parameter $\cE$ such that $S={\boldsymbol{\mathcal
R}}_{\Sigma}[\cE]$ is uniquely determined. It is also shown in
\cite[Proposition 5.9]{khy} that
    \begin{equation}\label{Si12}
    \Sigma_{12}(z)=B(z)\widehat{\Sigma}_{12}(z)
    \end{equation}
    where $\widehat{\Sigma}_{12}$ is a $*$-outer function in
    $\cS(\widetilde{\Delta}_*, \cY)$.
    {}From this identity and the definition of $\cM_{S,B}$
    we see directly that
    $\cM_{S,B}$ is contained in $M_B H^{2}_{\cW}$.  Secondly, we see from
the
    $*$-outer property of $\widehat\Sigma_{12}$ and the factorization
    \eqref{Si12} of $\Sigma_{12}$ that the operator of multiplication by
    $G(z) = \Sigma_{12}(z) (I - \cE(z) \Sigma_{22}(z))^{-1}$ is injective.
    Since we know that $M_G$ is a partial isometry, it now follows that
    ${M}_{G} \colon \cH(K_{\cE}) \to \cM_{S,B}$ is unitary.
    \end{proof}

    \begin{rem}
    {\rm One gets the same parametrization of ${\mathcal M}_{S,B}$ in case
    $S\in\cS(\cU,\cY)$ is not inner. }
    \label{R:1}
    \end{rem}

    As the following result indicates, spaces of the form ${\mathcal
    M}_{S,B}$ come up in the description of kernels of Toeplitz
    operators.  To formulate the result let us say that the triple
    $(S, B, \Gamma)$ is an {\em admissible triple} if
    \begin{enumerate}
	\item $S$ and $B$ in ${\mathcal S}(\cY)$  with $\cY$ finite-dimensional
	are {\em inner} (i.e.,
	$S$ and $B$ assume unitary values almost everywhere on the
	unit circle ${\mathbb T}$), and
	\item $\Gamma \in (H^{\infty}_{\cL(\cY)})^{\pm 1}$, i.e., both
	$\Gamma$ and $\Gamma^{-1}$ are in $H^{\infty}_{\cL(\cY)}$.
   \end{enumerate}
   We also need the following result from \cite{Barclay}.

   \begin{thm} \label{T:Barclay} (See \cite[Theorem
       4.1]{Barclay}.)
       Let $\epsilon > 0$ and suppose that $\cY \cong {\mathbb
       C}^{n}$ is a finite-dimensional coefficient Hilbert space.
       Suppose also that $\Phi_{u} \in
       L^{\infty}_{\cL(\cY)}$  has unitary values almost
       everywhere on ${\mathbb T}$.  Then there exists almost
       everywhere invertible functions $L,K \in
       H^{\infty}_{\cL(\cY)}$ with $L^{-1}, K^{-1}$ in
       $L^{\infty}_{\cL(\cY)}$ such that
       $$
       \Phi_{u} = L^{*}K \text{ almost everywhere on } {\mathbb T}
       $$
       and such that
       $$
        \| L\|_{\infty}, \|K\|_{\infty}, \|L^{-1}\|_{\infty},
	\|K^{-1}\|_{\infty} < 1 + \epsilon.
       $$
       \end{thm}

       For $\Phi$ a function in $L^{\infty}_{\cL(\cY)}$, the
       associated Toeplitz operator $T_{\Phi}$ on  $H^{2}_{\cL(\cY)}$
       is defined by
       $$
        T_{\Phi}(f) =  P_{H^{2}_{\cY}}(\Phi \cdot f).
       $$
       We consider such operators only for the case where $\Phi$ is
       invertible almost everywhere on the unit circle and in addition
       $\det \Phi^{*}\Phi$ is log-integrable:
       $$
       \int_{{\mathbb T}} \det\left( \Phi(\zeta)^{*} \Phi(\zeta)
       \right) \, |\tt{d}\zeta| > -\infty.
       $$

        We are now ready to state our result concerning Toeplitz
       kernels.  Here we use the notation $L^{\infty}_{\cL(\cY)}$ to
       denote the space of essentially uniformly bounded measurable
       $\cL(\cY)$-valued functions on the unit circle ${\mathbb T}$.

       \begin{thm} \label{T:Toeplitzkernel}
	   Let the coefficient Hilbert space $\cY$ be
	   finite-dimensional.
	   A subspace $\cM \subset H^{2}_{\cY}$ has the form $\cM =
	   \operatorname{Ker} T_{\Phi}$ for some $\Phi \in
	   L^{\infty}_{\cL(\cY)}$ with $\det \Phi^{*} \Phi$
	   log-integrable on ${\mathbb T}$ if and only if there is an
	   admissible triple $(S, B, \Gamma)$ so that $\cM$ has the
	   form
	   $$\cM = \Gamma B^{-1} \cdot \cM_{S,B}: = \Gamma B^{-1}
	   \cdot (\cK_{S} \cap M_{B} H^{2}_{\cY}).
	   $$
	   \end{thm}
	
	   \begin{proof}
	       Suppose that $\Phi \in L^{\infty}_{\cL(\cY)}$ with
	       $\det \Phi^{*} \Phi$ log-integrable.  Then there
	       exists an outer function $F \in H^{\infty}_{\cL(\cY)}$
	       solving the spectral factorization problem
	       $$
	       \Phi(\zeta) \Phi(\zeta)^{*} = F(\zeta)^{*} F(\zeta)
	       \text{ almost everywhere on } {\mathbb T}
	       $$
 (see e.g.~\cite{RR}).  If we set $\Phi_{u} : = F^{*-1} \Phi$, then
 $\Phi_{u}$ is unitary-valued on ${\mathbb T}$ and we have the
 factorization $\Phi = F^{*} \Phi_{u}$.  By Theorem \ref{T:Barclay},
 we may factor $\Phi_{u}$ as
 $$
    \Phi_{u} = L^{*} K
  $$
  with $L, K \in H^{\infty}_{\cL(\cY)}$ and $L^{-1}, K^{-1}\in L^{\infty}_{\cL(\cY)}$.  Let $L = L_{i}
  L_{o}$ and $K = K_{i} K_{o}$ be the inner-outer factorizations of
  $L$ and $K$ (again we refer to \cite{RR} for details on
  matrix-valued Hardy space theory).  Then $\Phi$ has the
  representation
  $$
   \Phi = F^{*} L_{o}^{*} L_{i}^{*} K_{i} K_{o}.
 $$

 Suppose now that $f \in H^{2}_{\cY}$ is in $\operatorname{Ker}
 T_{\Phi}$.  This condition can be equivalently written as
 $$
   F^{*} L_{o}^{*} L_{i}^{*} K_{i} K_{o} f \in H^{2 \perp}_{\cY},
 $$
or
$$
 L_{o}^{*} L_{i}^{*} K_{i} K_{o} f \in F^{* -1} H^{2 \perp}_{\cY}
 \cap L^{2}_{\cY}.
$$
Since $F^{-1}$ is an outer Nevanlinna-class function, it follows that
$F^{*-1} H^{2 \perp}_{\cY} \cap L^{2}_{\cY} = H^{2 \perp}_{\cY}$ and
we are left with
$$
L_{o}^{*} L_{i}^{*} K_{i} K_{o} f \in H^{2 \perp}_{\cY}.
$$
By a similar argument (even easier since $L_{o}^{-1} $ is bounded),
we deduce that, equivalently,
$\;   L_{i}^{*} K_{i} K_{0} f \in H^{2 \perp}_{\cY}$, or
$$
K_{i}K_{o} f \in L_{i} H^{2 \perp}_{\cY}.
$$
As $K_{i}K_{o}f \in H^{2}_{\cY}$, we actually have
\begin{equation}   \label{*}
 K_{i} K_{o} f \in L_{i} H^{2 \perp}_{\cY} \cap H^{2}_{\cY} =
 \cK_{L_{i}}.
\end{equation}
Clearly $K_{i}K_{o}f \in K_{i} H^{2}_{\cY}$ and hence \eqref{*} takes
the sharper form
$$
  K_{i} K_{o} f \in \cK_{L_{i}} \cap K_{i} H^{2}_{\cY} =: \cM_{L_{i},
  K_{i}}.
$$
Solving for $f$ gives
$$
   f \in K_{o}^{-1} K_{i}^{-1} \cM_{L_{i}, K_{i}} = \Gamma B^{-1}
   \cdot \cM_{S,B}
$$
where we set $(S, B, \Gamma)$ equal to the admissible triple $(L_{i},
K_{i}, K_{o}^{-1})$.  Conversely, all the steps are reversible:  if
$f \in K_{o}^{-1} K_{i}^{-1} \cM_{L_{i}, K_{i}}$, then $f \in
\operatorname{Ker} T_{\Phi}$.

Conversely, suppose that $(S, B, \Gamma)$ is any admissible triple.
Define $L_{o} \in (H^{\infty}_{\cL(\cY)})^{\pm 1}$ as any outer
solution of the spectral factorization problem
$$
 L_{o} L_{o}^{*}  = S^{*} B \Gamma^{*} \Gamma B^{*} S
$$
and set $\Phi_{u} = L_{o}^{*} S^{*} B \Gamma^{-1}$.  Then one can
check that $\Phi_{u}$ is even unitary-valued on ${\mathbb T}$ and
that $\operatorname{Ker} T_{\Phi_{u}} = \Gamma B^{-1} \cdot \cM_{S,
B}$.
\end{proof}

Theorem \ref{T:Toeplitzkernel} combined with Theorem \ref{T:homoint}
leads to the following Corollary, where the free-parameter space
$\cM_{S,B}$ in Theorem \ref{T:Toeplitzkernel} is replaced by the
arguably easier free-parameter space $\cH(K_{\cE})$.

    \begin{cor}  \label{C:Toeplitzkernel} Assume that the
	coefficient Hilbert space $\cY \cong {\mathbb C}^{n}$ has
	finite dimension.   A subspace $\cM
	\subset H^{2}_{\cY}$ is a
    Toeplitz kernel, i.e.,  $\cM = \operatorname{Ker}T_{\Phi}$, for an
$L^{\infty}_{\cL(\cY)}$-function $\Phi$ pointwise-invertible on ${\mathbb T}$
with $\det \Phi^{*} \Phi$ log-integrable
    if and only if  there is a
        function $\Gamma \in (H^{\infty}_{\cL(\cY)})^{\pm 1}$,
     inner functions $S$ and $B$ in $\cS(\cY)$, a function $\cE$ in the Schur
class $\cS(\cW, {\mathcal V})$ for some auxiliary Hilbert spaces $\cW$
and ${\mathcal V}$,
       and a function $G \colon {\mathbb D} \to \cL(\cW, {\mathbb C})$
       such that ${M}_{G} \colon g(z) \mapsto G(z) g(z)$ maps
       the de Branges-Rovnyak space $\cH(K_\cE)$ isometrically onto
       $ (H^{2} \ominus S \cdot H^{2}) \cap B \cdot H^{2}$, so that
       $$
       \cM = \Gamma B^{-1} G \cdot \cH(K_{\cE}).
       $$
       Here the function $G$ can be constructed explicitly from the pair
       $(S,B)$ by applying the construction in Theorem
\ref{T:homoint}.
      In particular, there exist auxiliary coefficient Hilbert spaces
       $\cW$ and ${\mathcal V}$ of dimension at most equal to
$\operatorname{dim}
       (H^{2} \ominus S \cdot H^{2}) + 1$ and a function $\Sigma = \left[
       \begin{smallmatrix} \Sigma_{11} & \Sigma_{12} \\ \Sigma_{21} &
           \Sigma_{22} \end{smallmatrix} \right] \in \cS(\cY
           \oplus \cW, \cY \oplus {\mathcal V})$ so that
       \begin{align*}
           & G(z) = \Sigma_{12}(z) (I - \cE(z) \Sigma_{22}(z))^{-1}  \\
           & S(z) = \Sigma_{11}(z) + \Sigma_{12}(z) (I - \cE(z)
           \Sigma_{22}(z))^{-1} \cE(z) \Sigma_{21}(z).
        \end{align*}
        \end{cor}

        \begin{proof}  Simply plug in the representation of a space
	    $\cM_{S,B}$ in Theorem \ref{T:homoint} into the
	    parametrization of
	    $\operatorname{Ker} T_{\Phi}$ in Theorem
	    \ref{T:Toeplitzkernel}.
       \end{proof}

       \begin{rem} \label{R:nearlyinv} {\em
           A subspace $\cM$ of $H^{2}$ is said to be {\em nearly
invariant}
           for the backward shift $M_{z}^{*}$ if
           $f(z)/z \in \cM$ whenever $f \in \cM$ and $f(0) = 0$.
           In \cite{Hitt}, Hitt obtained the following characterization of
           almost invariant subspaces:  {\em a subspace $\cM \subset
H^{2}$
           is nearly invariant if and only if there is an inner
	   function $u$ with $u(0)  = 0$
    and a holomorphic function $g$ on the disk ${\mathbb D}$
	   so that
           \begin{equation}  \label{nearinvM}
           \cM = M_g \cdot (H^{2} \ominus M_u H^{2})
           \end{equation}
           where $g$ is such that
           the multiplication operator $M_{g} \colon h(z) \mapsto g(z)h(z)$
           acts isometrically   from $H^{2} \ominus M_u H^{2}$ into
$H^{2}$}.
            Theorem 0.3 from \cite{FtHK09}  characterizes
           which functions $g$ are such that $M_{g}$ acts contractively
from
           $H^2\ominus M_u H^{2}$ into $H^{2}$ for a given inner function
           $u$ with $u(0)=0$:  such a $g$ must have the form
           \begin{equation} \label{g-rep}
               g(z) = a_{1}(z) (1 - u(z) b_{1}(z))^{-1}
          \end{equation}
          for a function $\sigma(z) = \left[\begin{smallmatrix}a_1(z)
          \\b_1(z)\end{smallmatrix} \right]$
          in the Schur class $\cS({\mathbb C}, {\mathbb C}^{2})$.  It is
not hard to see
          that  $M_{g} \colon H^{2} \ominus M_u H^{2} \to H^{2}$ is
isometric exactly when in addition
          $$
           |a_{1}(\zeta)|^{2} + |b_1(\zeta)|^{2} = 1 \quad\text{for almost
all}\quad
           \zeta \in {\mathbb T}
          $$
          from which it follows also that $\|g\|_{2} = 1$.

          If one starts
          with $g \in H^{2}$ of unit norm for which $M_{g} \colon H^{2}
          \ominus M_u H^{2}\to H^2$ is isometric, one can construct the
          representation \eqref{g-rep} for $g$ as follows.  Let $g$ have
          inner-outer factorization $g = \omega \cdot f$ with
	  $\omega$ inner
          and $f$ outer with $f(0)>0$. Let $F$ denote the Herglotz
integral
          of $|f|^{2}$, i.e., for $z
           \in {\mathbb D}$ we set
           $$
            F(z) = \int_{{\mathbb T}} \frac{ \zeta + z}{ \zeta - z}
            |f(\zeta)|^{2}\,  \frac{ |d\zeta|}{2 \pi}.
           $$
           The fact that $\|g\|_{2}^{2} = \| f\|^{2}_{2} = 1$ implies that
            $F(0) = 1$; we also note that $F(z)$ has positive real part
for $z$ in
           ${\mathbb D}$.  If we then set $b = \frac{F-1}{F+1}$, then $b$
is
           in the unit ball of $H^{\infty}$ and satisfies $b(0) = 0$.
           The fact that $M_{g}$ is isometric from $H^{2} \ominus M_u
H^{2}$
           into $H^{2}$ forces $b$ to be divisible by $u$, so we can
factor
           $b$ as $b = u b_{1}$ with $b_{1}$ in the unit ball of
           $H^{\infty}$. Let $a$ be the unique outer function with
           $|a(\zeta)|^{2} = 1 - |b(\zeta)|^{2}$ for almost all $\zeta \in
           {\mathbb T}$ and with $a(0)>0$.  Set $a_{1}(z) = \omega(z)
a(z)$.
          Then $g$ has the representation \eqref{g-rep} with this choice
of
          $a_{1}$ and $b_{1}$.  The characterization of isometric
          multipliers from $H^{2} \ominus M_u H^{2}$ into $H^{2}$ in this
          form together with the application to Hitt's theorem is one of
the
          main results of Sarason's paper \cite{SarasonOT35}.
          A direct proof for the special case where $u(z) = z$ appears in
          \cite[Lemma 2 page 488]{SarasonOT41} in connection with a
          different problem, namely, the characterization of Nehari pairs.

          Following the terminology of \cite{Kheifets-gamma}, we say that
          pair of $H^{\infty}$
           functions $(a,b)$ is a {\em $\gamma$-generating pair} if
           \begin{enumerate}
               \item[(i)] $a$ and $b$ are functions in the unit ball of
               $H^{\infty}$,
               \item[(ii)] $a$ is outer and $a(0) > 0$,
               \item [(iii)] $b(0) = 0$, and
               \item[(iv)] $|a|^{2} + |b|^{2} =1$ almost everywhere on the
               unit circle ${\mathbb T}$.
         \end{enumerate}
         Note that the pair of functions $(a,b)$ appearing above in the
          representation \eqref{g-rep} (with $a_{1} = \omega a$ and $b = u
         b_{1}$) is $\gamma$-generating.

         It is not hard to see that the kernel of any bounded Toeplitz
         operator $\operatorname{Ker} T_{\phi} \subset H^{2}$ with $\phi
\in
         L^{\infty}$ is always nearly invariant; hence any Toeplitz kernel
         $\cM = \operatorname{Ker} T_{\phi}$ is in particular
         of the form \eqref{nearinvM} as described above.  The result of
         Hayashi in \cite{Hay90} is the following characterization of
which
         nearly invariant subspaces are Toeplitz kernels: {\em the
subspace
         $\cM \subset H^{2}$ is the kernel of some bounded Toeplitz
operator
         $T_{\phi}$ if and only if $\cM$ has the form \eqref{nearinvM}
with
         $\omega(z) = 1$ for some $\gamma$-generating pair and inner
         function $u$ with $u(0) = 0$ subject to the additional condition
         that the function $\left( \frac{a}{1 - z \overline{u}} b
         \right)^{2}$ is an exposed point of the unit ball of $H^{1}$.}
	 These results have now been extended to the matrix-valued
	 case in \cite{CCP} and \cite{Chevrot}.

         The paper \cite{dy} of Dyakonov obtains the alternative characterization of
         Toeplitz kernels given in Theorem \ref{T:Toeplitzkernel} for
	 the scalar case; our proof is a simple adaptation of the
	 proof in \cite{dy} to the matrix-valued case, with the
	 matrix-valued factorization result from \cite{Barclay}
	 (Theorem \ref{T:Barclay} above) replacing the special
	 scalar-valued version of the result due to Bourgain
	 \cite{Bourgain}.
	 The advantage of this characterization of Toeplitz kernels
	 (as opposed to  the earlier results of Hayashi \cite{Hay86}
	 for the scalar case and of Chevrot \cite{Chevrot} for the
	 matrix-valued case)
	 is the avoidance of mention of $H^{1}$-exposed points
         (as there is no useable characterization of such objects).
         Moreover Dyakonov formulates his results for
          subspaces of $H^{p}$  rather than just $H^{2}$; we expect
	  that our Theorem  \ref{T:Toeplitzkernel} extends in the
	  same way to the $H^{p}$ setting, but we do not pursue this
	  generalization here as Theorem \ref{T:homoint} is at
	  present formulated only for the $H^{2}$ setting.
         Note that our characterization of Toeplitz kernels
	 (Corollary \ref{C:Toeplitzkernel} above)
         brings us back to the formulations of Hayashi and Sarason for
         characterizations of nearly invariant subspaces/Toeplitz kernels
         in two respects:
         (1) the characterization involves a multiplication operator
         which is unitary from some model space of functions to the space
         to be characterized, and (2) there is an explicit
	 parametrization of
         which such multipliers have this unitary property.
         }\end{rem}

\section{Boundary interpolation}\label{S:BI}
    \setcounter{equation}{0}

    In this section we consider a boundary interpolation problem in
    a de Branges-Rovnyak space $\cH(K_S)$. For the sake of simplicity we
    focus on the scalar-valued case; it is a routine exercise to
    extend the results presented here to the matrix- or
    operator-valued case by using the notation and machinery from
    \cite{bk3, bk6, bk5}. In what follows,
    $f_j(z)=\frac{f^{(j)}(z)}{j!}$ stands for the $j$-th Taylor
    coefficient at
    $z\in\D$ of an analytic function $f$. By $f_j(t_0)$ we denote the
    boundary limit
    \begin{equation}   \label{bdrylim}
    f_j(t_0):=\lim_{z\to t_0}f_j(z)
    \end{equation}
    as $z$ tends to a boundary point $t_0\in\T$ nontangentially, provided
    the limit exists and is finite.

    The next theorem collects from the existing literature several equivalent
    characterizations of the higher order Carath\'eodory-Julia condition for a
    Schur-class function $s\in\cS$
    \begin{equation}
    \liminf_{z\to t_0}\frac{\partial^{2n}}{\partial
    z^{n}\partial\bar{z}^{n}} \, \frac{1-|s(z)|^2}{1-|z|^2}<\infty,
    \label{6.2}
    \end{equation}
    where now $z$ tends to $t_0\in\T$ unrestrictedly in $\D$.

    \begin{thm}\label{T:6.1}
    Let $s\in\cS$, $t_0\in\T$ and $n\in{\mathbb N}$. The following are
    equivalent:
    \begin{enumerate}

    \item $s$ meets the Carath\'eodory-Julia condition \eqref{6.2}.

    \item The function
    $\frac{\partial^n}{\partial\bar{\zeta}^n}K_S(\cdot,\zeta)$ stays
bounded in
    the norm of $\cH(K_S)$ as $\zeta$ tends radially to $t_0$.

    \item It holds that
    \begin{equation} \label{6.3}
     {\displaystyle\sum_k\frac{1-|a_k|^2}{|t_0-a_k|^{2n+2}}+
    \int_0^{2\pi}\frac{d\mu(\theta)}{|t_0-e^{i\theta}|^{2n+2}}<\infty}
    \end{equation}
    where the numbers $a_k$ come from the Blaschke product of the inner-outer factorization of $s$:
    $$
    s(z)=\prod_k\frac{\bar{a}_k}{a_k}\cdot\frac{z-a_k}{1-z\bar{a}_k}\cdot
    \exp\left\{-\int_0^{2\pi}\frac{e^{i\theta}+z}{e^{i\theta}-z}d
    \mu(\theta)\right\}.
   $$
\item The boundary limits $s_j:=s_j(t_0)$ exist for $j=0,\ldots,n$ and the
    functions
    \begin{equation}
    K_{t_0,j}(z):=\frac{z^j}{(1-z\overline{t}_0)^{j+1}}-s(z)\cdot
\sum_{\ell=0}^j\frac{z^{j-\ell}\overline{s}_j}{(1-z\overline{t}_0)^{j+1-\ell}}\quad (j=0,\ldots,n)
    \label{6.4}
    \end{equation}
    belong to $\cH(K_s)$.

    \item The boundary limits $s_j:=s_j(t_0)$ exist for $j=0,\ldots,n$ and
the
    function $K_{t_0,n}(z)$ defined via formula \eqref{6.4} belongs to
    $\cH(K_s)$.

    \item The boundary limits $s_j:=s_j(t_0)$ exist for $j=0,\ldots,2n+1$
and
    are such that $|s_0|=1$ and the matrix
    \begin{equation}
    {\mathbb P}^s_{n}(t_0):=\left[\begin{array}{ccc} s_1 & \cdots &
    s_{n+1} \\ \vdots & &\vdots \\
    s_{n+1} & \cdots & s_{2n+1}\end{array}\right]{\bf \Psi}_n(t_0)
    \left[\begin{array}{ccc}\overline{s}_0 & \ldots &
    \overline{s}_n\\ & \ddots & \vdots \\ 0 &&\overline{s}_0
    \end{array}\right]
    \label{6.5}
    \end{equation}
    is Hermitian, where the first factor is a Hankel matrix, the
    third factor is an upper triangular Toeplitz matrix and where
    ${\bf \Psi}_n(t_0)$ is the upper triangular matrix given by
    \begin{equation}
    {\bf\Psi}_n(t_0)=\left[\Psi_{j\ell}\right]_{j,\ell=0}^n,\quad
    \Psi_{j\ell}=(-1)^{\ell}\left(\begin{array}{c} \ell \\ j
    \end{array}\right)t_0^{\ell+j+1},\quad 0\le j\leq\ell\le n.
    \label{6.6}
    \end{equation}

    \item For every $f\in\cH(K_s)$, the boundary limits
    $f_j(t_0)$ exist for $j=0,\ldots,n$.
    \end{enumerate}
    Moreover, if one of the conditions (1)--(7) is satisfied, and hence all, then:
    \begin{enumerate}
    \item[(a)] The matrix \eqref{6.5} is positive semidefinite and
    equals
    \begin{equation}
    {\mathbb P}^s_{n}(t_0)=\left[\langle K_{t_0,i}, \,
    K_{t_0,j}\rangle_{\cH(K_s)}\right]_{i,j=0}^n.
    \label{6.7}
    \end{equation}
    \item[(b)] The functions \eqref{6.4} are boundary reproducing kernels
    in $\cH(K_s)$ in the sense that
    \begin{equation}
    \langle f, \, K_{t_0,j}\rangle_{\cH(K_s)}=f_j(t_0):=\lim_{z\to
    t_0}\frac{f^{(j)}(z)}{j!}\quad\mbox{for}\quad j=0,\ldots,n.
    \label{6.8}
    \end{equation}
    \end{enumerate}
    \end{thm}

    \begin{proof}
    Equivalences (1)$\Longleftrightarrow$(4)$\Longleftrightarrow$(5), implication
    (5)$\Longrightarrow$(6) and statements (a) and (b)
    were proved in \cite{bk3}; implication (6)$\Longrightarrow$(1)
    and equivalence (1)$\Longrightarrow$(7) appear in \cite{bk4} and
    \cite{bk5}, respectively. Equivalence (1)$\Longleftrightarrow$(7)
    was established in \cite{acl} for $s$ inner and extended in \cite{fm}
    to general Schur class functions. Equivalence (2)$\Longleftrightarrow$(7)
    was shown in \cite[Section VII]{sarasonsubh}.
    \end{proof}

    Theorem \ref{T:6.1} suggests a
    boundary interpolation problem for functions in ${\cH(K_s)}$ with data set
    \begin{equation}
    {\fD}_b=\{s, {\bf t}, {\bf k}, \{f_{ij}\}\},
    \label{6.9}
    \end{equation}
    consisting of two tuples ${\bf t}=(t_1,\ldots,t_k)\in\T^k$ and ${\bf
    n}=(n_1,\ldots,n_k)\in\N^k$, a doubly indexed sequence $\{f_{ij}\}$
    (with $0\le j\le n_i$ and $1\le i\le k$) of complex numbers and of
    a Schur-class function $s$ subject to the Carath\'eodory-Julia conditions
    \begin{equation}
    \liminf_{z\to t_i}\frac{\partial^{2n_i}}{\partial
    z^{n_i}\partial\bar{z}^{n_i}} \, \frac{1-|s(z)|^2}{1-|z|^2}<\infty
    \quad\mbox{for}\quad i=1,\ldots,k,
    \label{6.10}
    \end{equation}
    or one of the equivalent conditions from Theorem \ref{T:6.1}.

    We consider the problem
    ${\bf BP}_{\cH(K_s)}$: {\em Given a date set \eqref{6.9} satisfying \eqref{6.10}, find all
    $f\in\cH(K_s)$
    such that $\|f\|_{\cH(K_s)}\le 1$ and
    \begin{equation}
    f_j(t_i):=\lim_{z\to
    t_i}\frac{f^{(j)}(z)}{j!}=f_{ij}\quad\mbox{for}\quad
    j=0,\ldots,n_i\quad\mbox{and}\quad i=1,\ldots,k.
    \label{6.11}
    \end{equation}}

    According to Theorem \ref{T:6.1}, conditions \eqref{6.10} guarantee that
    all the boundary limits in \eqref{6.11} exist as well as the boundary
    limits
    \begin{equation}
    s_{ij}:=s_j(t_i)\quad\mbox{for}\quad
    j=0,\ldots,2n_i+1\quad\mbox{and}\quad i=1,\ldots,k.
    \label{6.12}
    \end{equation}
    We let $N=\sum_{i=1}^k(n_i+1)$ denote
    the total number of interpolation conditions in \eqref{6.11} and we
    let $\cX=\C^N$. With the data set \eqref{6.9}, we associate the matrices
    \begin{equation}
    T=\left[\begin{array}{ccc}T_1 & & 0 \\ & \ddots & \\
    0 && T_k\end{array}\right]\quad\mbox{and}\quad
    \left[\begin{array}{c}E \\ N \\ \bx^*\end{array}\right]=
    \left[\begin{array}{cccc}E_1 & E_2 &\ldots &
    E_k\\ N_1 & N_2&\ldots & N_k\\ \bx_1^* & \bx_2^*&\ldots
    &\bx_k^*\end{array}\right],
    \label{6.13}
    \end{equation}
    where
    \begin{equation}
    T_i=\left[\begin{array}{cccc} \bar{t}_i & 1 & \ldots & 0 \\
    0 &  \bar{t}_i & & \vdots \\ \vdots & \ddots & \ddots& 1\\
    0 & \ldots & 0 & \bar{t}_i\end{array}\right]\quad\mbox{and}\quad
    \left[\begin{array}{c}E_i \\ N_i \\ \bx_i^*\end{array}\right]=
    \left[\begin{array}{cccc} 1& 0& \ldots&0 \\ \overline{s}_{i,0}&
    \overline{s}_{i,1}& \ldots & \overline{s}_{i,n_i}\\
    \overline{f}_{i,0}&
    \overline{f}_{i,1}& \ldots & \overline{f}_{i,n_i}
    \end{array}\right].
    \label{6.14}
    \end{equation}
    Now we define the function $\Fs$ by formula \eqref{1.17} with
    $E$, $T$ and $N$ given by \eqref{6.13} and \eqref{6.14}, and show that
    $\Fs$ can be expressed in terms of boundary kernels as
    \begin{equation}
    \Fs(z):=(E-s(z)N)(I-zT)^{-1}=\begin{bmatrix}{\bf K}_{t_1,n_1}(z)
    &\ldots & {\bf K}_{t_k,n_k}(z) \end{bmatrix}
    \label{6.15}
    \end{equation}
    where
    \begin{equation}
    {\bf K}_{t_i,n_i}(z)=\begin{bmatrix}
    K_{t_i,0}(z) & K_{t_i,1}(z) & \ldots & K_{t_i,n_i}(z)\end{bmatrix}
    \quad\mbox{for}\quad i=1,\ldots,k,
    \label{6.16}
    \end{equation}
    and where the functions $K_{t_i,j}$ are the boundary kernels
    defined via formula \eqref{6.4}. Indeed, it follows from definitions
    \eqref{6.14} that
    \begin{equation}
    \left[\begin{array}{c} E_{i} \\ N_i
    \end{array}\right]\left(I-zT_i\right)^{-1}=
    \left[\begin{array}{ccc}{\displaystyle\frac{1}{1-z\bar{t}_i}} &
    \ldots & {\displaystyle\frac{z^{n_i}}{(1-z\bar{t}_i)^{n_i+1}}}\\
    {\displaystyle\frac{\overline{s}_{i,0}}{1-z\bar{t}_i}} & \ldots &

    {\displaystyle\sum_{\ell=0}^{n_i}\frac{\overline{s}_{i,\ell}z^{n_i-\ell}}
    {(1-z\bar{t}_i)^{n_i+1-\ell}}}\end{array}\right].
    \label{6.17}
    \end{equation}
    Multiplying the latter equality by $\begin{bmatrix} 1 &
    -s(z)\end{bmatrix}$ on the left and taking into account \eqref{6.16}
    and explicit formulas \eqref{6.4} for $K_{t_i,j}$ we obtain
    \begin{equation}
    (E_i-s(z)N_i)\left(I-zT_i\right)^{-1}={\bf K}_{t_i,n_i}(z),
    \label{6.18}
    \end{equation}
    and equality \eqref{6.15} now follows from the block structure
    \eqref{6.13} of $T$, $E$ and $N$.

    Now we will show that the problem {\bf AIP}$_{\cH(K_S)}$
    with the  $\{s, T, E, N, \bx\}$ taken in the form \eqref{6.13},
    \eqref{6.14} is equivalent to the problem {\bf BP}$_{\cH(K_s)}$. We first
    check that the data is {\bf AIP}-admissible.

    The first requirement is self-evident
    since all the eigenvalues of $T$ fall onto the unit circle and therefore
    $(I-zT)^{-1}$ is a rational functions with no poles inside $\D$.
    However, it is worth noting that the pair $(E,T)$ is not output-stable
    and so {\bf BP}$_{\cH(K_s)}$ cannot be embedded into the scheme of the
    problem ${\bf OAP}_{\cH(K_s)}$ of Section \ref{S:OAP}.
    To verify that the requirements (2) concerning $\Fs$ are also fulfilled,
    we first observe from \eqref{6.15} and \eqref{6.16} that
    for a generic vector
    $x=\operatorname{Col}_{1\le i\le k}\operatorname{Col}_{0\le j\le n_i}
    x_{ij}$ in $\cX$,
    \begin{equation}
    \Fs(z)x=\sum_{i=1}^k\sum_{j=0}^{n_i}K_{t_i,j}(z)x_{ij}.
    \label{6.18a}
    \end{equation}
    Now it follows from statement (3) in Theorem \ref{T:6.1} that
    the operator $M_{F^s}$ maps $\cX$ into $\cH(K_s)$.

    Furthermore, due to \eqref{6.18a}, the operator
    $P=M_{F^s}^{[*]}M_{F^s}$ admits the following block
matrix
    representation with respect to the standard basis of $\cX=\C^N$:
    \begin{equation}
    P=\left[P_{ij}\right]_{i,j=1}^k\quad\mbox{where}\quad
    P_{ij}=\left[\langle K_{t_i,\ell}, \,
    K_{t_j,r}\rangle_{\cH(K_s)}\right]_{\ell=0,\ldots,n_i}^{r=0,\ldots,n_j}
    \label{6.21}
    \end{equation}
    and the explicit formulas for $P_{ij}$ in  terms of boundary limits
    \eqref{6.12} are (see \cite{bk6} for details):
    \begin{equation}
    P_{ij} =H_{ij}\cdot {\bf
    \Psi}_{n_j}(t_j)\cdot\left[\begin{array}{ccc}\overline{s}_{j,0} &
\ldots &
    \overline{s}_{j,n_j}\\ & \ddots & \vdots \\ 0 &&\overline{s}_{j,0}
    \end{array}\right]
    \label{6.24}
    \end{equation}
    where ${\bf \Psi}_{n_j}(t_j)$ is defined via formula \eqref{6.6},
where
    $H_{ii}=\left[s_{i,\ell+r+1}\right]_{\ell,r=1}^{n_i}$ is a Hankel
    matrix and where the matrices
    $H_{ij}$ (for $i\neq j$)
    are defined entry-wise by
    \begin{eqnarray}
    \left[H_{ij}\right]_{r, m}&=&
    \sum_{\ell=0}^{r} (-1)^{r-\ell}
    \left(\begin{array}{c}m+r-\ell \\
    m\end{array}\right)\frac{s_{i,\ell}}
    {(t_i-t_j)^{m+r-\ell+1}}\nonumber \\
    &&-\sum_{\ell=0}^{m} (-1)^{r}\left(\begin{array}{c}m+r-\ell \\
    r\end{array}\right)\frac{s_{j,\ell}}{(t_i-t_j)^{m+r-\ell+1}}
    \label{6.25}
    \end{eqnarray}
    for $r=0,\ldots,n_i$ and $m=0,\ldots,n_j$. It was shown in
    \cite{bk6} that the matrix $P$ of the above structure satisfies
    the Stein identity \eqref{1.11}, with $T$, $E$ and $N$ given by
    \eqref{6.13}, \eqref{6.14}, whenever $P$ is Hermitian. This works
since
    in the present situation, $P$ is positive semidefinite. Thus, the data set
    $\{s, T, E, N, \bx\}$ is {\bf AIP}-admissible.

    By the reproducing property \eqref{6.8}, representation \eqref{6.18}
    implies that for every $f\in\cH(K_s)$,
    $$
    \langle M_{F^s}^{[*]}f, \; x\rangle_{\cX}=
    \langle f, \;
    M_{F^s}
x\rangle_{\cH(K_s)}=\sum_{i=1}^k\sum_{j=0}^{n_i}f_j(t_i)\overline{x}_{ij}.
    $$
    On the other hand, for $\bx$ defined in \eqref{6.15} and \eqref{6.16},
    $$
    \langle \bx, \,
    x\rangle_{\cX}=\sum_{i=1}^k\sum_{j=0}^{n_i}f_{i,j}\overline{x}_{ij}.
    $$
    It follows from the two last equalities that interpolation conditions
    \eqref{6.11} are equivalent to the equality
    $$
    \langle M_{F^s}^{[*]}f, \; x\rangle_{\cX}=\langle \bx, \,
    x\rangle_{\cX}
    $$
    holding for every $x\in\cX$, i.e., the equality
    $M_{F^s}^{[*]}f=\bx$ holds. We now conclude that the problem {\bf
    AIP}$_{\cH(K_S)}$
    with the data set $\{s, T, E, N, \bx\}$ taken in the form \eqref{6.13},
    \eqref{6.14} is equivalent to the {\bf BP}$_{\cH(K_s)}$. Thus,
    {\em the problem {\bf BP}$_{\cH(K_s)}$ has a solution if and only if $P\ge {\bf
    x}{\bf x}^*$ where $P$ is defined in terms of boundary limits
\eqref{6.12}
    for $s$ as in \eqref{6.24} and where ${\bf x}$ is defined in
\eqref{6.13},
    \eqref{6.14}}. If this is the case, {\em the solution set for the problem
    {\bf BP}$_{\cH(K_s)}$ is parametrized as in Theorem
\ref{T:AIPsol}}.

\bibliographystyle{amsplain}
\providecommand{\bysame}{\leavevmode\hbox to3em{\hrulefill}\thinspace}

\end{document}